\documentclass[11pt,a4paper]{article}
\usepackage[T1]{fontenc}
\usepackage[utf8]{inputenc}
\usepackage{authblk}
\usepackage{geometry}
\usepackage{color}
\usepackage{cite}
\usepackage{amsmath}
\usepackage{graphicx,wrapfig}
\usepackage{amssymb,amsgen,color}
\usepackage{amsfonts}
\usepackage{hyperref}
\usepackage{multirow}
\usepackage{epsfig,epstopdf,color,bm}
\usepackage{booktabs}
\usepackage{verbatim}
\usepackage{multicol}
\usepackage{makecell}
\usepackage{amsthm}

\usepackage{subcaption}

\newtheorem{theorem}{Theorem}[section]
\newtheorem{lemma}[theorem]{Lemma}

\newcommand{\eps}{\varepsilon}
\newcommand{\nn}{\nonumber}

\newtheorem{remark}[theorem]{Remark}

\numberwithin{equation}{section}

\author[1]{Yue Feng}
\affil[1]{Laboratoire Jacques-Louis Lions, Sorbonne Universit\'e, Paris 75005, France}

\author[1]{Georg Maierhofer}
\author[1]{Katharina Schratz}
\date{}

\begin{document}

\title{Long-time error bounds of low-regularity integrators for nonlinear Schr\"odinger equations}

\maketitle
%    \subjclass is required.
%\subjclass[2020]{Primary 35Q41,35Q55, 65M15, 65M70}

%%%%% Begin Abstract %%%%%%%%%%%
\begin{abstract}
\noindent
We introduce a new non-resonant low-regularity integrator for the cubic nonlinear Schr\"odinger equation (NLSE) allowing for long-time error estimates which are optimal in the sense of the underlying PDE. The main idea thereby lies in treating the zeroth mode exactly within the discretization. For long-time error estimates, we rigorously establish the error bounds of different low-regularity integrators for the NLSE with small initial data characterized by a dimensionless parameter $\eps \in (0, 1]$. We begin with the low-regularity integrator for the quadratic NLSE in which the integral is computed exactly and the improved uniform first-order convergence in $H^r$ is proven at $O(\eps \tau)$ for solutions in $H^r$ with $r > 1/2$ up to the time $T_{\eps} = T/\eps$ with fixed $T > 0$. Then, the improved uniform long-time error bound is extended to a symmetric second-order low-regularity integrator in the long-time regime. For the cubic NLSE, we design new non-resonant first-order and symmetric second-order low-regularity integrators which treat the zeroth mode exactly and rigorously carry out the error analysis up to the time $T_{\eps} = T/\eps^2$.  With the help of the regularity compensation oscillation (RCO) technique, the improved uniform error bounds are established for the new non-resonant low-regularity schemes, which further reduce the long-time error by a factor of $\eps^2$ compared with classical low-regularity integrators for the cubic NLSE. Numerical examples are presented to validate the error estimates and compare with the classical time-splitting methods in the long-time simulations.
\end{abstract}
%%%%% end %%%%%%%%%%%

%%%%%% Keywords %%%%%%%

%%%%%% maketitle %%%%%%
\maketitle

%%%%%% Introduction %%%%%%

\section{Introduction}
We consider the following quadratic nonlinear Schr\"odinger equation (NLSE) with small initial data on the one-dimensional (1D) torus $\mathbb{T} = (-\pi, \pi)$ \cite{BT,KPV,Kish,Ozawa}
\begin{equation}
\left\{
\begin{aligned}
&i\partial_t u(x, t) = -\partial^2_x u(x, t) + u^2(x, t), \quad x \in \mathbb{T}, \ t>0,	\\
&u(x, 0) = \eps \phi(x), \quad x \in \mathbb{T},
\end{aligned}\right.
\label{eq:QNLSE1}
\end{equation}
and the cubic NLSE \cite{BC,CW,Sasaki,Sun}
\begin{equation}
\left\{
\begin{aligned}
&i\partial_t u(x, t) = -\partial^2_x u(x, t) + |u(x, t)|^2 u(x, t), \quad x \in \mathbb{T}, \ t>0,	\\
&u(x, 0) = \eps \phi(x), \quad x \in \mathbb{T},
\end{aligned}\right.
\label{eq:CNLSE1}
\end{equation}
where $i = \sqrt{-1}$, $u(x, t) \in \mathbb{C}$ is the complex wave function/order parameter with the spatial coordinate $x$ and time $t$, $\phi(x)$ is a given complex-valued function, and $\eps \in (0, 1]$ is a dimensionless parameter to characterize the size of the initial data. For notational simplicity, we focus entirely on the 1D case in this paper, but we note that the constructions and analysis extend verbatim to the study of low-regularity schemes for the NLSE considered on the $d$-dimensional torus $\mathbb{T}^d$.

The nonlinear Schr\"odinger equation (NLSE) as a canonical dispersive partial differential equation (PDE) plays a significant role in many fields including physics, chemistry, biology and engineering \cite{ADK,PS,SS}. The specific forms of the nonlinearity are related to diverse applications. Specially, the quadratic NLSE appears in nonlinear optics for the optical material with a $\chi^{(2)}$ (i.e., quadratic) nonlinear response, laser-plasma interactions and wave propagation in nonlinear fibers \cite{BTST,CCO2,CMS}. The cubic NLSE, also called Gross--Pitaevskii equation (GPE), is widely used to describe the Bose-Einstein condensate (BEC) \cite{BC,PS,Sch}. In the analytical aspect, a great deal of effort has been put into the Cauchy problem for the NLSE including the existence and uniqueness of the solution as well as the asymptotic behaviour in time \cite{BJ,CW,GV,Ozawa}. Recently, increasing work has been devoted to the study of the long-time behaviour of nonlinear evolution equations and the effect of the nonlinearity to the solutions \cite{BT,CaS,FO}. The lifespan of the semilinear Schr\"odinger equation with small initial data has attracted much interest, which shows that the lifespan of the quadratic NLSE \eqref{eq:QNLSE1} and the cubic NLSE \eqref{eq:CNLSE1} is at least $O(\eps^{-1})$ and  $O(\eps^{-2})$, respectively \cite{BT,II,Oh,Sun}. 

Rescaling the amplitude of the variable by introducing $w(x, t) = u(x, t)/\eps$, we have the following quadratic NLSE with $O(\eps)$-nonlinearity and $O(1)$-initial data
\begin{equation}
\left\{
\begin{aligned}
&i\partial_t w(x, t) = -\partial^2_x w(x, t) + \eps w^2(x, t), \quad  x \in \mathbb{T}, \ t>0,	\\
&w(x, 0) = \phi(x), \quad  x \in \mathbb{T},
\end{aligned}\right.
\label{eq:QNLSE2}
\end{equation}
and the cubic NLSE with $O(\eps^2)$-nonlinearity and $O(1)$-initial data
\begin{equation}
\left\{
\begin{aligned}
&i\partial_t w(x, t) = -\partial^2_x w(x, t) + \eps^2 |w(x, t)|^2 w(x, t), \quad  x \in \mathbb{T}, \ t>0,	\\
&w(x, 0) = \phi(x), \quad  x \in \mathbb{T}.
\end{aligned}\right.
\label{eq:CNLSE2}
\end{equation}
In fact, the long-time dynamics of the quadratic NLSE \eqref{eq:QNLSE2} with $O(\eps)$-nonlinearity and $O(1)$-initial data is equivalent to that of the NLSE \eqref{eq:QNLSE1} with $O(1)$-nonlinearity and $O(\eps)$-initial data. The equivalence of the long-time dynamics for the cubic NLSEs \eqref{eq:CNLSE1} and \eqref{eq:CNLSE2} is similar.

Over the past decades, many accurate and efficient numerical methods have been proposed to solve the NLSE including the finite difference methods, exponential integrators, time-splitting methods \cite{ABB,CCO1,DFP,HO,MQ}. According to these error estimates, time-splitting methods and exponential integrators for the NLSE  require two additional derivatives, i.e., $\phi \in H^{r+2}(\mathbb{T}^d)$ with $r > d/2$ to get first-order convergence in $H^{r}(\mathbb{T}^d)$ and finite difference methods require even more regularity of the initial data. For nonsmooth initial data, a low-regularity exponential-type integrator was proposed to get first-order convergence  in $H^{r}(\mathbb{T}^d)$ for $\phi \in H^{r+1}(\mathbb{T}^d)$ in the cubic case and no additional regularity is required for the quadratic case \cite{OS}. Later, filtered low-regularity schemes were proposed which allow for low regularity estimates thanks to discrete Bourgain spaces and discrete Strichartz type estimates \cite{LW,ORS,WY}. The low-regularity schemes are also called resonance-based schemes due to the construction around Fourier based expansions of the solution and the underlying resonant structure of the equation \cite{BS}. In addition, the resonance-based schemes are introduced to approximate the dynamics of other dispersive PDEs such as the KdV equation and the nonlinear Klein--Gordon equation \cite{CS,HS,MS} with a framework shown in \cite{BS}. However, existing error analysis is so far restricted to finite time $T = O(1)$, see \cite{CS,HS,MS,OS} and the references therein. Deserve to be mentioned, a symplectic low-regularity integrator for the KdV equation was presented that preserves the geometric structure of the problem \cite{MS}. The NLSE is time reversible or symmetric, i.e., it is unchanged under the change of variable in time as $t \to -t$ and taken conjugate in the equation \cite{ABB}. It is important to apply the symmetric scheme to discretize the NLSE such that the symmetric property is still valid at the discrete level. Also, the geometric property is essential to control the long-time behaviour of numerical schemes \cite{CHL,WZ}. 

In the long-time regime, the behaviours of different numerical schemes for the NLSE on the compact domain have been studied \cite{CCMM,FGP1,FGP2,GL1,GL2}. For the time-splitting method, the improved uniform error bound was proven at $O(\eps^2\tau^2)$ for the long-time dynamics of the cubic NLSE up to the time of order $O(1/\eps^2)$ via the regularity compensation oscillation (RCO) technique with the regularity assumption $u \in H^5(\mathbb{T})$ \cite{BCF}.  However, as far as we know, there is no work on the long-time error estimates of the low-regularity schemes for the NLSE up to now. The aim of this paper is to carry out the long-time error bounds on different low-regularity schemes for the NLSE. For the quadratic NLSE, we focus on the classical first-order and symmetric second-order low-regularity integrators and rigorously establish the improved uniform error bounds up to the time $T_{\eps} = T/\eps$ with fixed $T > 0$. However, for the cubic NLSE, the classical low-regularity schemes just obtain uniform error bounds up to the time of order $O(1/\eps^2)$. Note the appropriate long-time regime for the cubic NLSE is at the time of order $O(1/\eps^2)$ as opposed to the time of order $O(1/\eps)$ (the latter regime is natural for the quadratic case, cf. \eqref{eq:QNLSE2} vs. \eqref{eq:CNLSE2}). Indeed, some of the most successful low-regularity schemes (cf. \cite{OS,BS}) are based on an iteration of Duhamel's formula in the twisted variable $v=\exp(-it\partial_x^2)u$ which leads to an approximation of the form
\begin{align*}
v(t_n+\tau)\approx \sum_{l \in \mathbb{Z}}e^{ilx}\!\!\!\!\!\!\!\sum_{\substack{l_1, l_2, l_3 \in \mathbb{Z} \\ l = -l_1+l_2+l_3}}\!\!\! \underbrace{e^{i t_n(l^2+l_1^2-l_2^2-l_3^2)}}_{=:\mathcal{F}_1}\underbrace{\int^{\tau}_0\!\! e^{is(l^2+l_1^2-l_2^2-l_3^2)} ds}_{=:_{\mathcal{F}_2}} \overline{\hat{v}_{l_1}(t_n)}\hat{v}_{l_2}(t_n)\hat{v}_{l_3}(t_n),
\end{align*}
where $\hat{v}_l$ are the Fourier coefficients of the unknown function $v$. The low-regularity integrator in \cite{OS} is then constructed from this expression by approximating $\exp(is(l^2+l_1^2-l_2^2-l_3^2))\approx \exp(2isl_1^2)$. From our analysis based on the regularity compensation oscillation (RCO) technique in section~\ref{sec:improved_bound_NRL1}, we deduce that summation of the factors $\mathcal{F}_1$ leads to cancellation of their error contribution in the long-time regime whenever their phase is non-zero, i.e. whenever $l^2+l_1^2-l_2^2-l_3^2\neq 0$.

As a result, at the time of order $O(1/\eps^2)$ the dominant contribution to the error arises from the zeroth mode, i.e. the terms when $l^2+l_1^2-l_2^2-l_3^2=0$. Our novel idea is then to treat this particular mode exactly, meaning we do not approximate the integral in $\mathcal{F}_2$ for which we have the exact expression $\int_0^\tau 1 ds=\tau$. This idea leads to improved performance in the long-time regime, and allows us in particular to prove improved uniform error bounds using the RCO technique. Further details on this construction can be found in section~\ref{sec:design_non_resonant_schemes} but let us briefly mention a numerical example shown in Figure~\ref{fig:cubic_err_fn_of_epsilon_introduction}, which highlights how our non-resonant scheme (``NRLI1'') leads to significantly improved long-time behaviour in comparison to state-of-the-art resonance-based low-regularity integrator for the cubic NLSE \cite{OS} (``Ostermann \& Schratz '18''). In this particular experiment we take initial data in $H^2$ and a fixed time step $\tau=0.05$ for various values of $\eps$. Further detailed numerical results are provided in section~\ref{sec:numerical_results}.

\begin{figure}[h!]
	\centering
	\includegraphics[width=0.5\textwidth]{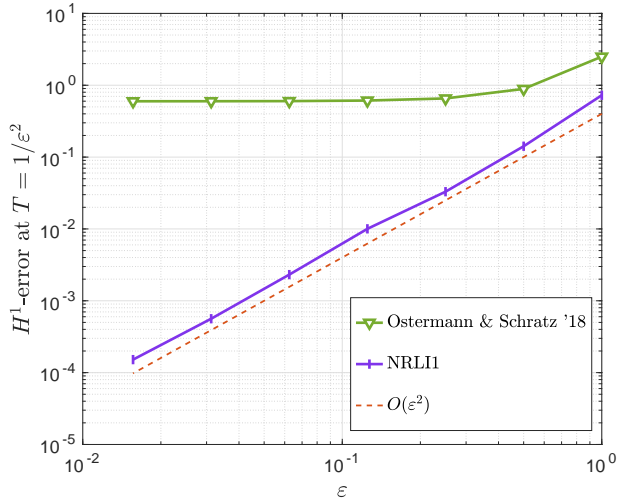}
	\caption{Comparison of long-time error of low-regularity schemes.}
	\label{fig:cubic_err_fn_of_epsilon_introduction}
\end{figure}

The rest of the paper is organized as follows. In section 2, we present the first-order and symmetric second-order low-regularity integrators for the quadratic NLSE and rigorously establish the improved uniform error bounds up to the time $T_{\eps} = T/\eps$. In section 3, we design new non-resonant first-order and symmetric second-order low-regularity schemes for the cubic NLSE and carry out the error bounds up to the time $T_{\eps} = T/\eps^2$. Numerical results for the long-time dynamics of  the NLSE are shown in section 4. Finally, some conclusions are drawn in section 5. Throughout this paper, the notation $A \lesssim B$ is used to represent that there exists a generic constant $C>0$ independent of the time step $\tau$ and $\eps$ such that $|A| \leq C B$. We always let $r > 1/2$ and denote by $\|\cdot\|_r$ the standard $H^r = H^r(\mathbb{T})$ Sobolev norm and exploit the well-known bilinear estimate 
\begin{equation}
\left\|fg\right\|_r	\leq C_{r}\left\|f\right\|_r\left\|g\right\|_r,
\label{eq:bi}
\end{equation}
which holds for a constant $C_r > 0$.

\section{Long-time error bounds for quadratic NLSE}
In this section, we recall the first-order and symmetric second-order low-regularity integrators for the quadratic NLSE \eqref{eq:QNLSE2} and establish the improved uniform error bounds up to the time $T_{\eps} = T/\eps$. 

For $f \in L^2(\mathbb{T})$, we denote its Fourier expansion by $f(x) = \sum_{l \in \mathbb{Z}}\hat{f}_l e^{ilx}$. We define a regularization of $\partial^{-1}_x$ through its action in Fourier space by 
\begin{equation}
(\partial^{-1}_x)_l := \left\{
\begin{aligned}
& (il)^{-1}, \quad l \neq 0,	\\
& 0,  \qquad\quad  l = 0,
\end{aligned}\right.
\quad {\rm i.e.,} \quad \partial^{-1}_x f(x) = \sum_{l \in \mathbb{Z}\backslash \{0\}} \frac{1}{il}\hat{f}_l e^{ilx}.
\end{equation}
For the so-called $\varphi_1$ function defined as
\begin{equation}
\varphi_1(z) = \frac{e^z-1}{z},
\end{equation}
by continuity, we have
\begin{equation}
\left(\frac{e^{it\partial^2_x}-1}{it\partial^2_x}\right)_l = \left\{
\begin{aligned}
& \frac{e^{-itl^2}-1}{-itl^2}, \quad l \neq 0,	\\
& 1,  \qquad\qquad \quad l = 0.
\end{aligned}\right.
\end{equation}

\subsection{A first-order low-regularity integrator}
As presented in \cite{OS}, we introduce the twisted variable $v(t) = e^{-it\partial^2_x}w(t)$, which satisfies the following quadratic NLSE 
\begin{equation}
\left\{
\begin{aligned}
&i\partial_t v(x, t) = \eps e^{-it\partial^2_x}\left( e^{it\partial^2_x}v(x, t)\right)^2, \quad x \in \mathbb{T}, \ t>0,	\\
&v(x, 0) = \phi(x), \quad x \in \mathbb{T}.
\end{aligned}\right.
\label{eq:QNLSE_v}
\end{equation}
In the following of this paper, we denote by $v(t) := v(x, t)$ in short, i.e., omit the spatial variable when there is no confusion.  By Duhamel's formula, the exact solution of $v(t)$ is given as
\begin{equation}
v(t_n+\tau) = v(t_n) -i\eps\int^{\tau}_0 e^{-i(t_n+s)\partial^2_x}\left( e^{i(t_n+s)\partial^2_x}v(t_n+s)\right)^2ds,	
\label{eq:Duh}
\end{equation}
where $t_n  = n\tau$ for $ n = 0, 1, \ldots, \frac{T/\eps}{\tau}-1$. Since $e^{it\partial^2_x}$ is a linear isometry on $H^r$ for $r > 1/2$ and all $t \in \mathbb{R}$, we have 
\begin{align}
\left\|v(t_n+\tau) - v(t_n)\right\|_r & \leq \eps\int^{\tau}_0 \left\|v(t_n+s)\right\|_r^2 ds \leq \eps \tau \sup_{0 \leq s \leq \tau} \left\|v(t_n+s)\right\|_r^2. 
\end{align}
In the mild solution \eqref{eq:Duh}, we use $v(t_n)$ to approximate $v(t_n+s)$ for $0\leq s\leq \tau$, which leads to 
\begin{equation}
v(t_n+\tau) \approx v(t_n) -i\eps\int^{\tau}_0 e^{-i(t_n+s)\partial^2_x}\left( e^{i(t_n+s)\partial^2_x}v(t_n)\right)^2ds.
\label{eq:fo1}
\end{equation}
Then, we write the integral in Fourier space as
\begin{align}
I_1^{\tau}(v, t_n) & = \int^{\tau}_0 e^{-i(t_n+s)\partial^2_x}\left( e^{i(t_n+s)\partial^2_x}v\right)^2ds \nn \\
& = \int^{\tau}_0 e^{-i(t_n+s)\partial^2_x}\left[\left( e^{i(t_n+s)\partial^2_x}\sum_{l_1}\hat{v}_{l_1}e^{il_1x}\right)\left( e^{i(t_n+s)\partial^2_x}\sum_{l_2}\hat{v}_{l_2}e^{il_2x}\right)\right]ds \nn \\
& = \int^{\tau}_0 \sum_{l_1, l_2}e^{i(t_n+s)[(l_1+l_2)^2-l_1^2-l_2^2]} \hat{v}_{l_1}\hat{v}_{l_2}e^{i(l_1+l_2)x}ds.
\label{eq:I}	
\end{align}
The key relation in the construction of the scheme is 
\begin{equation}
(l_1+l_2)^2 - l_1^2 - l_2^2 = 2l_1l_2,	
\label{eq:freq}
\end{equation}
which allows us to derive the first-order low-regularity scheme without any additional regularity assumption on the exact solution.
Substituting the relation \eqref{eq:freq} into the integral \eqref{eq:I}, we obtain
\begin{align}
I_1^{\tau}(v, t_n) = &\ \sum_{\substack{l_1, l_2 \\ l_1, l_2 \neq 0}}e^{it_n[(l_1+l_2)^2-l_1^2-l_2^2]} \frac{e^{i\tau [(l_1+l_2)^2-l_1^2-l_2^2] }- 1}{2il_1l_2} \hat{v}_{l_1}\hat{v}_{l_2}e^{i(l_1+l_2)x} \nn\\
& \ + 2\tau\hat{v}_0\sum_{ l_1\neq 0}\hat{v}_{l_1}e^{il_1x} +\tau \hat{v}_0^2 \nn\\
= & \ \frac{i}{2}\sum_{\substack{l_1, l_2 \\ l_1, l_2\neq 0}}e^{it_n[(l_1+l_2)^2-l_1^2-l_2^2]} \frac{e^{i\tau [(l_1+l_2)^2-l_1^2-l_2^2] }- 1}{(il_1)(il_2)} \hat{v}_{l_1}\hat{v}_{l_2}e^{i(l_1+l_2)x} \nn\\
& \ + 2\tau\hat{v}_0\sum_{ l_1\in \mathbb{Z}}\hat{v}_{l_1}e^{il_1x} - \tau \hat{v}_0^2 \nn\\
= & \ \frac{i}{2} e^{-it_n\partial^2_x}\left[e^{-i\tau\partial^2_x}\left(e^{i(t_n+\tau)\partial^2_x} \partial^{-1}_x v\right)^2 - \left(e^{it_n\partial^2_x} \partial^{-1}_x v\right)^2 \right] +2\tau\hat{v}_0 v -\tau\hat{v}_0^2.
\end{align}
Denoting $v^n$ as the approximation of $v(t_n)$ and combining with the approximation \eqref{eq:fo1}, we have the first-order low-regularity scheme for $v(t)$ as
\begin{align}
v^{n+1} = &\ \Phi^{\tau}_{t_n}(v^n) \nn\\
:= &\ \left(1-2i\eps\tau\hat{v}^n_0\right)v^n + i\eps \tau (\hat{v}^n_0)^2\nn\\
& \ + \frac{\eps}{2} e^{-it_n\partial^2_x}\left[e^{-i\tau\partial^2_x}\left(e^{i(t_n+\tau)\partial^2_x} \partial^{-1}_x v^n\right)^2 - \left(e^{it_n\partial^2_x} \partial^{-1}_x v^n\right)^2 \right]. 
\label{eq:v_flow}
\end{align}
Finally, we twist the variable back, i.e., $w^n = e^{it\partial^2_x}v^n$, then the first-order low-regularity integrator (LI1) is given as
\begin{align}
w^{n+1} =&\ \left(1-2i\eps\tau\hat{w}^n_0\right)e^{i\tau\partial^2_x}w^n + i\eps \tau (\hat{w}^n_0)^2 \nn\\
&\ + \frac{\eps}{2}\left[\left(e^{i\tau\partial^2_x} \partial^{-1}_x w^n\right)^2 - e^{i\tau\partial^2_x}\left(\partial^{-1}_x w^n\right)^2 \right]. 	
\label{eq:LI1}
\end{align}
\begin{remark}
Similarly, the low-regularity integrator was introduced for the following quadratic NLSE \cite{OS}
\begin{equation}
\left\{
\begin{aligned}
&i\partial_t w(x, t) = -\partial^2_x w(x, t) + \eps |w(x, t)|^2, \quad  x \in \mathbb{T}, \ t>0,	\\
&w(x, 0) = \phi(x), \quad  x \in \mathbb{T}.
\end{aligned}\right.
\label{eq:QNLSE3}
\end{equation}
By the same idea, the first-order low-regularity scheme is 
\begin{align}
\label{eq:LI_con}
w^{n+1} = & \ \left(1-i\eps\tau\overline{\hat{w}}^n_0\right)e^{i\tau\partial^2_x}w^n - i\eps \tau \left\|w^n\right\|^2_{L^2} \nn\\
& \ + \frac{\eps}{2}\partial^{-1}_x\left[\left(e^{i\tau\partial^2_x} w^n\right)\left(e^{-i\tau\partial^2_x} \partial^{-1}_x \overline{w}^n\right) - e^{i\tau\partial^2_x}\left(w^n\partial^{-1}_x \overline{w}^n\right) \right]. 	
\end{align}
\end{remark}

\subsection{Improved uniform error bound of the LI1 \eqref{eq:LI1}}
In this subsection, we are going to establish the improved uniform error bound on the first-order low-regularity integrator (LI1) \eqref{eq:LI1} up to the time $T_{\eps} = T/\eps$ with fixed $T>0$.

\begin{theorem}
\label{thm:w1}
Let $r > 1/2$ and assume that the exact solution of the quadratic NLSE \eqref{eq:QNLSE2} satisfies $w(t) \in H^r$ for $0 \leq t\leq T/\eps$. Then there exists a constant $\tau_0>0$ such that for $0 < \tau \leq \tau_0$, we have the following error bound on the LI1 \eqref{eq:LI1} as
\begin{equation}
\left\|w(t_n) - w^n\right\|_r \lesssim \eps\tau,	\quad 0 \leq n \leq \frac{T/\eps}{\tau}.
\end{equation}
\end{theorem}

\begin{remark}
The improved uniform error bound on the first-order low regularity integrator \eqref{eq:LI_con} for the quadratic NLSE \eqref{eq:QNLSE3} is still valid. We just show the proof of the LI1 \eqref{eq:LI1} and omit the detailed proof of the scheme \eqref{eq:LI_con} for brevity.
\end{remark}

Due to the linear isometry of $e^{it\partial^2_x}$, we have 
\begin{equation}
\left\|w(t_n) - w^n\right\|_r = \left\|e^{it_n \partial^2_x}\left(v(t_n) - v^n\right)\right\|_r = \left\|v(t_n) - v^n\right\|_r,
\end{equation}
then it remains to prove the error bound for $v(t_n)$ as follows.

\begin{theorem}
\label{thm:v1}
Let $r > 1/2$ and assume that the exact solution of the quadratic NLSE  \eqref{eq:QNLSE_v} satisfies $v(t) \in H^r$ for $0 \leq t\leq T/\eps$. Then there exists a constant $\tau_0>0$ such that for $0 < \tau \leq \tau_0$, we have the following error bound for the first-order low-regularity scheme \eqref{eq:v_flow}
\begin{equation}
\left\|v(t_n) - v^n\right\|_r \lesssim \eps\tau,	\quad 0 \leq n \leq \frac{T/\eps}{\tau}.
\label{eq:Q1v_bound}
\end{equation}
\end{theorem}
Before proving the error bound \eqref{eq:Q1v_bound}, we prepare some results for the stability and local truncation error of the first-order low-regularity scheme \eqref{eq:v_flow}.
\begin{lemma}
(Stability) For $f, g \in H^r$ with $r > 1/2$, we have 
\begin{equation}
\left\|\Phi^{\tau}_t(f)-\Phi^{\tau}_t(g)\right\|_r \leq \left(1+\eps L\tau\right)	\left\|f - g\right\|_r,
\end{equation}
where $L$ depends on $\left\|f+g\right\|_r$.
\label{lem:stability}
\end{lemma}
\begin{proof}
According to the numerical flow \eqref{eq:v_flow}, we have
\begin{align}
\Phi^{\tau}_t(f)-\Phi^{\tau}_t(g) =& \  \left(f-g\right)- 2i\eps\tau\left(\hat{f}_0f - \hat{g}_0g\right) + i\eps \tau \left(\hat{f}_0^2-\hat{g}_0^2\right)\nn\\
&\ +\frac{\eps}{2} e^{-it_n\partial^2_x}\left[e^{-i\tau\partial^2_x}\left(e^{i(t_n+\tau)\partial^2_x} \partial^{-1}_x f\right)^2 - \left(e^{it_n\partial^2_x} \partial^{-1}_x f\right)^2 \right]\nn\\
&\ -\frac{\eps}{2} e^{-it_n\partial^2_x}\left[e^{-i\tau\partial^2_x}\left(e^{i(t_n+\tau)\partial^2_x} \partial^{-1}_x g\right)^2 - \left(e^{it_n\partial^2_x} \partial^{-1}_x g\right)^2 \right],
\end{align}
which immediately implies
\begin{equation}
\left\|\Phi^{\tau}_t(f)-\Phi^{\tau}_t(g)\right\|_r  \leq \left(1+\eps L\tau\right)\left\|f - g\right\|_r,
\end{equation}
where $L$ depends on $\left\|f+g\right\|_r$.
\end{proof}

\begin{lemma}
(Local truncation error) Let $r > 1/2$ and assume the exact flow of the quadratic NLSE \eqref{eq:QNLSE_v} is $v(t_k + t) = \phi^t(v(t_k)) \in H^{r}$ for $0 \leq t \leq \tau$. Then, we have the following estimates for the local truncation error
\begin{equation*}
\left\|\phi^{\tau}(v(t_k))-\Phi^{\tau}_{t_k}(v(t_k))\right\|_r \leq c \eps^2\tau^2, 	
\end{equation*}
where $c$ depends on $\sup_{0 \leq t \leq T/\eps}\left\|\phi^t(v(t_k))\right\|_{r}$.
\label{lem:leb}
\end{lemma}
\begin{proof}
By the expression of the exact solution \eqref{eq:Duh} and the approximation \eqref{eq:fo1}, we can write the local truncation error as
\begin{align*}
& \phi^{\tau}(v(t_k))-\Phi^{\tau}_{t_k}(v(t_k)) \\
& \ = -i\eps\int^{\tau}_0 e^{-i(t_n+s)\partial^2_x}\left[\left( e^{i(t_n+s)\partial^2_x}v(t_n+s)\right)^2-\left( e^{i(t_n+s)\partial^2_x}v(t_n)\right)^2\right]ds.
\end{align*}
Combining with the bilinear estimate \eqref{eq:bi}, it implies
\begin{align*}
&\left\|\phi^{\tau}(v(t_k))-\Phi^{\tau}_{t_k}(v(t_k))\right\|_r \\
&\  \leq \eps \int^{\tau}_0\left\|\left( e^{i(t_n+s)\partial^2_x}v(t_n+s)\right)^2-\left( e^{i(t_n+s)\partial^2_x}v(t_n)\right)^2\right\|_r ds \nn \\
&\  \leq c\eps\tau \sup_{0\leq s \leq \tau}\left( \left\|\left(v(t_n+s)+v(t_n)\right)\right\|_r \left\|\left(v(t_n+s)-v(t_n)\right)\right\|_r\right)\nn\\
&\  \leq c\eps^2\tau \sup_{0\leq s \leq \tau} \int^{s}_0 \left\|\left( e^{i(t_n+\xi)\partial^2_x}v(t_n+s)\right)^2\right\|_r ds\nn\\
&\  \leq c\eps^2\tau^2,
\end{align*}
where $c$ depends on $\sup_{0 \leq t \leq T/\eps}\left\|\phi^t(v(t_k))\right\|_{r}$.
\end{proof}

\noindent
\begin{proof}[{Proof of Theorem \ref{thm:v1}}] By Lemma \ref{lem:stability} and Lemma \ref{lem:leb}, it leads to
\begin{align*}
\left\|v(t_{k+1}) - v^{k+1}\right\|_r	& = \left\|\phi^{\tau}(v(t_{k})) - \Phi^{\tau}_{t_k}(v^k)\right\|_r	\\
& \leq  \left\|\phi^{\tau}(v(t_{k})) - \Phi^{\tau}_{t_k}(v(t_k))\right\|_r +  \left\| \Phi^{\tau}_{t_k}(v(t_k))-\Phi^{\tau}_{t_k}(v^k)\right\|_r\\
& \leq c\eps^2\tau^2+ \left(1+\eps L\tau\right)	\left\|v(t_k) - v^k\right\|_r.
\end{align*}
Denoting $e^k = v(t_k)-v^{k}$, we find
\begin{equation}
\|e^{k+1}\|_r - \|e^k\|_r \leq c\eps^2\tau^2+\eps L\tau\|e^k\|_r.
\end{equation}
Since $e^0 = 0$, summing up above inequality for $k =0, 1, \ldots, n$, we obtain
\begin{equation}
\|e^{n+1}\|_r \leq \|e^0\|_r + \eps L\tau \sum_{k=0}^n\|e^k\|_r + c\eps^2\tau^2(n+1)\leq c\eps T\tau+\eps L\tau \sum_{k=0}^n\|e^k\|_r  .
\end{equation}
Applying Gronwall inequality, we have
\begin{equation}
\|e^{n+1}\|_r  \leq c\eps T\tau e^{LT}, \quad 0 \leq n \leq \frac{T/\eps}{\tau}-1,
\end{equation}
which completes the proof of the error bound \eqref{eq:Q1v_bound}. 
\end{proof}

\subsection{A symmetric second-order low-regularity integrator}
In the previous section, we establish the improved uniform error bound for the first-order low-regularity scheme (LI1). Compared with the classical methods (e.g., time-splitting method and exponential integrator), it not only obtains improved uniform error bound but also better approximations at low regularity. However, the above low-regularity scheme destroys the symmetric structure of the quadratic NLSE \eqref{eq:QNLSE2}. In order to overcome this drawback, we approximate \eqref{eq:Duh} in the following form
\begin{align}
v(t_n+\tau) \approx v(t_n) -\frac{i\eps}{2}\int^{\tau}_0 e^{-i(t_n+s)\partial^2_x}\Bigg[& \left(e^{i(t_n+s)\partial^2_x}v(t_n)\right)^2 \nn\\
& \ +\left(e^{i(t_n+s)\partial^2_x}v(t_n+\tau)\right)^2\Bigg]ds.
\label{eq:fo2}
\end{align}
which leads to the following symmetric second-order low-regularity scheme
\begin{align}
v^{n+1} := &\ v^n - i\eps\tau\left(\hat{v}^n_0 v^n + \hat{v}^{n+1}_0 v^{n+1}\right)+ \frac{i\eps \tau}{2}\left((\hat{v}^n_0)^2 +  (\hat{v}^{n+1}_0)^2\right)\nn\\
& \ + \frac{\eps}{4} e^{-it_n\partial^2_x}\Bigg[e^{-i\tau\partial^2_x}\left(\left(e^{i(t_n+\tau)\partial^2_x} \partial^{-1}_x v^n \right)^2 + \left(e^{i(t_n+\tau)\partial^2_x} \partial^{-1}_x v^{n+1} \right)^2 \right)\nn\\
& \ \qquad\qquad\qquad - \left(e^{it_n\partial^2_x}\partial^{-1}_x v^n \right)^2 - \left(e^{it_n\partial^2_x}\partial^{-1}_x v^{n+1} \right)^2 \Bigg]. 
\label{eq:v_flow2}
\end{align}
Twisting the variable back to $w^n$, we obtain the symmetric second-order low-regularity integrator (SLI2) is given as
\begin{align}
w^{n+1} = & \ e^{i\tau\partial^2_x}w^n  - i\eps\tau\left(\hat{w}^n_0 e^{i\tau\partial^2_x}w^n + \hat{w}^{n+1}_0 w^{n+1} \right)+ \frac{i\eps \tau}{2}\left((\hat{w}^n_0)^2 +  (\hat{w}^{n+1}_0)^2\right) \nn\\
&\ + \frac{\eps}{4}\Bigg[\left(e^{i\tau\partial^2_x} \partial^{-1}_x w^n\right)^2 + \left( \partial^{-1}_x w^{n+1}\right)^2 \nn\\
& \ \qquad \quad  - e^{i\tau\partial^2_x}\left(\left(\partial^{-1}_x w^n\right)^2 + \left(e^{-i\tau\partial^2_x}\partial^{-1}_x w^{n+1}\right)^2 \right)\Bigg]. 	
\label{eq:LI2}
\end{align}
\begin{remark}
Let the numerical scheme de denoted by $w^{n+1} = \Phi_{\tau}(w^n)$, then it is easy to check that the method is symmetric in the sense that $\Phi_{\tau} = \Phi_{-\tau}^{-1}$.	
\end{remark}

\begin{remark}
For the quadratic NLSE \eqref{eq:QNLSE3}, we can use the same idea to obtain the symmetric second-order low-regularity integrator
\begin{align}
w^{n+1} = & \ e^{i\tau\partial^2_x}w^n - i\eps\tau\left(\overline{\hat{w}}^n_0 e^{i\tau\partial^2_x}w^n + \overline{\hat{w}}^{n+1}_0 w^{n+1}\right) - \frac{i\eps \tau}{2} \left(\left\|w^n\right\|^2_{L^2}+ \left\|w^{n+1}\right\|^2_{L^2}\right) \nn\\
& \ + \frac{\eps}{4}\partial^{-1}_x\Bigg[\left(e^{i\tau\partial^2_x} w^n\right)\left(e^{-i\tau\partial^2_x} \partial^{-1}_x \overline{w}^n\right) + w^{n+1}\partial^{-1}_x \overline{w}^{n+1}\nn\\
& \qquad \qquad \quad - e^{i\tau\partial^2_x}\left(w^n\partial^{-1}_x \overline{w}^n + \left(e^{-i\tau\partial^2_x}w^{n+1}\right)\left(e^{i\tau\partial^2_x}\partial^{-1}_x \overline{w}^{n+1}\right)\right) \Bigg]. 	
\end{align}
\end{remark}

\subsection{Improved uniform error bound of the SLI2 \eqref{eq:LI2}} In this subsection, we establish the improved uniform error bound of the SLI2 \eqref{eq:LI2}. For simplicity, we just show the main result.
\begin{theorem}
\label{thm:w2}
Assume that the exact solution of quadratic NLSE  \eqref{eq:QNLSE2} satisfies $w(t) \in H^{r+2}$ for $0 \leq t\leq T/\eps$. Then there exists a constant $\tau_0>0$ such that for $0 < \tau \leq \tau_0$, we have the following error bound on the SLI2 \eqref{eq:LI2} as
\begin{equation}
\left\|w(t_n) - w^n\right\|_r \lesssim \eps\tau^2,	\quad 0 \leq n \leq \frac{T/\eps}{\tau}.
\end{equation}
\end{theorem}
As before, it suffices to prove the error bound on $v^n$ given by \eqref{eq:v_flow2} as follows.
\begin{theorem}
\label{thm:v2}
Assume that the exact solution of the quadratic NLSE \eqref{eq:QNLSE_v} satisfies $v(t) \in H^{r+2}$ for $0 \leq t\leq T/\eps$. Then there exists a constant $\tau_0>0$ such that for $0 < \tau \leq \tau_0$, we have the following error bound
\begin{equation}
\left\|v(t_n) - v^n\right\|_r \lesssim \eps\tau^2, \quad 0 \leq n \leq \frac{T/\eps}{\tau}.
\label{eq:Q2_bound}
\end{equation}
\end{theorem}
For simplicity, we refer to  \cite{BYA,BMS,MS} for the stability and local truncation error of the symmetric scheme.
\begin{lemma}
(Stability) Denote the map of the SLI2 \eqref{eq:v_flow2} by
\begin{equation}
v^{n+1} = \Psi^{\tau}_{t_n}(v^n),	
\end{equation}
then we have for $f, g \in H^r$,
\begin{equation}
\left\|\Psi^{\tau}_t(f)-\Psi^{\tau}_t(g)\right\|_r \leq \exp(\eps L\tau)	\left\|f - g\right\|_r,
\end{equation}
where $L$ depends on $\left\|f\right\|_r$ and $\left\|g\right\|_r$.
\label{lem:stability2}
\end{lemma}

\begin{lemma}
(Local truncation error) Assume the exact flow of the quadratic NLSE \eqref{eq:QNLSE_v} is $v(t_k + t) = \phi^t(v(t_k)) \in H^{r+2}$ for $0 \leq t \leq \tau$. Then, we have the following error estimates for the local truncation error
\begin{equation*}
\left\|\phi^{\tau}(v(t_k))-\Psi^{\tau}_{t_k}(v(t_k))\right\|_r \leq c \eps^2\tau^3, 	
\end{equation*}
where $c$ depends on $\sup_{0 \leq t \leq T/\eps}\left\|\phi^t(v(t_k))\right\|_{r+2}$.
\label{lem:leb2}
\end{lemma}
\noindent
{\emph{Proof of Theorem \ref{thm:v2}}.} According to Lemma \ref{lem:stability2} and Lemma \ref{lem:leb2}, we have
\begin{align*}
\left\|v(t_{k+1}) - v^{k+1}\right\|_r	& = \left\|\phi^{\tau}v(t_{k}) - \Phi^{\tau}_{t_k}(v^k)\right\|_r	\\
& \leq  \left\|\phi^{\tau}v(t_{k}) - \Phi^{\tau}_{t_k}(v(t_k))\right\|_r +  \left\| \Phi^{\tau}_{t_k}(v(t_k))-\Phi^{\tau}_{t_k}(v^k)\right\|_r\\
& \leq c\eps^2\tau^2+ \exp(\eps L\tau)	\left\|v(t_k) - v^k\right\|_r.
\end{align*}
Denoting $e^k = v(t_k)-v^{k}$, we find
\begin{equation}
\|e^{k+1}\|_r - \|e^k\|_r \leq c\eps^2\tau^2+\left(\exp(\eps L\tau)-1\right)\|e^k\|_r.
\end{equation}
Summing above inequality for $k =0, 1, \ldots, n$, we obtain
\begin{align*}
\|e^{n+1}\|_r & \leq \|e^0\|_r + \left(\exp(\eps L\tau)-1\right) \sum_{k=0}^n\|e^k\|_r + c\eps^2\tau^3(n+1)  \\
&\leq c\eps T\tau^2 +\left(\exp(\eps L\tau)-1\right)\sum_{k=0}^n\|e^k\|_r .
\end{align*}
Applying Gronwall inequality, we have
\begin{equation}
\|e^{n+1}\|_r  \leq c\eps \tau^2, \quad 0 \leq n \leq \frac{T/\eps}{\tau}-1,
\end{equation}
where $c$ is different from that in \eqref{lem:leb2} which completes the proof of the error bound \eqref{eq:Q2_bound}.

\section{Long-time error bounds for cubic NLSE}
For the first-order low-regularity scheme, the integral $I_1^{\tau}(v, t_n)$ is computed exactly with the help of the relation \eqref{eq:freq}. This idea is difficult to extend to other nonlinearities, including the cubic case, because the exact integral no longer has a representation in physical coordinates leading to prohibitively large cost in the implementation (since FFT is no longer available). For reasons that will become apparent below, the resonance-based scheme constructed in \cite{OS} however only leads to uniform error bounds of order $O(\tau)$ for the long-time dynamics up to time $T_{\eps}=T/\eps^2$ for the cubic NLSE with $O(\eps)$-initial data and $O(1)$-nonlinearity. This stands in contrast to prior work which allowed us, based on our introduction of the regularity compensation oscillation (RCO) technique, to establish an improved uniform error bound of order $O(\eps^2\tau^2)$ for the second-order time-splitting method (Strang splitting) for the time of order $O(1/\eps^2)$ even in the cubic NLSE case \cite{BCF}. In this section, combining the idea of the low-regularity exponential integrator with the RCO technique, we design new non-resonant low-regularity exponential integrators, which can not only deal with the rough initial data but also obtain improved uniform error bound up to the time of order $O(1/\eps^2)$ for the smooth initial data.

\subsection{A non-resonant first-order low-regularity integrator}\label{sec:design_non_resonant_schemes}
As before, we again consider the twisted variable $v(t) = e^{-it\partial^2_x}w(t)$ satisfying the following equation 
\begin{equation}
\left\{
\begin{aligned}
&i\partial_t v(t) = \eps^2 e^{-it\partial^2_x}\left[|e^{it\partial^2_x}v(t)|^2e^{it\partial^2_x}v(t)\right], \quad x \in \mathbb{T}, \ t>0,	\\
&v(x, 0) = \phi(x), \quad x \in \mathbb{T},
\end{aligned}\right.
\label{eq:CNLSE_v}
\end{equation}
with mild solution given by
\begin{equation}
v(t_n+\tau) = v(t_n) -i\eps^2\int^{\tau}_0 e^{-i(t_n+s)\partial^2_x}\left[|e^{i(t_n+s)\partial^2_x}v(t_n+s)|^2 e^{i(t_n+s)\partial^2_x}v(t_n+s)\right]ds.
\label{eq:Duh_2}
\end{equation}
In the numerical simulation, it remains to approximate the integral
\begin{equation}
\mathcal{I}_{{\rm cub}}^{\tau}(v, t_n) := \int^{\tau}_0 e^{-i(t_n+s)\partial^2_x}\left[|e^{i(t_n+s)\partial^2_x}v(t_n+s)|^2 e^{i(t_n+s)\partial^2_x}v(t_n+s)\right]ds.	
\end{equation}
The construction of the first-order scheme in \cite{OS} is based on the following approximation inside this integral
\begin{equation}
v(t_n + s) \approx v(t_n),	
\end{equation}
which yields, once the integral is expressed in terms of the Fourier expansion,
\begin{align*}
I_{{\rm cub},1 }^{\tau}(v, t_n) =   \sum_{l \in \mathbb{Z}}e^{ilx}\sum_{\substack{l_1, l_2, l_3 \in \mathbb{Z} \\ l = -l_1+l_2+l_3}}e^{i t_n(l^2+l_1^2-l_2^2-l_3^2)}\overline{ \hat{v}_{l_1}}  \hat{v}_{l_2}\hat{v}_{l_3}\int^{\tau}_0 e^{is(l^2+l_1^2-l_2^2-l_3^2)} ds,
\end{align*}
with $\hat{v}_{l_j} = \hat{v}_{l_j}(t_n)$ for $j = 1, 2, 3$. Then, the remaining integral is approximated by
\begin{equation}
	\int^{\tau}_0 e^{is(l^2+l_1^2-l_2^2-l_3^2)} ds \approx \int^{\tau}_0 e^{2isl_1^2} ds.\end{equation}
As shown in \cite{OS} this approximation leads to excellent low-regularity convergence properties of the resulting scheme. However, this approximation leads to artificial resonances in the zeroth mode (i.e. when $l^2+l_1^2-l_2^2-l_3^2=0$) which negatively affect the long-time behaviour of the method. As such, we introduce a new low-regularity scheme based on a special treatment of the zeroth mode. For the resonant case $l^2+l_1^2-l_2^2-l_3^2 = 0$, it can be computed exactly as
\begin{equation}
	\int^{\tau}_0 e^{is(l^2+l_1^2-l_2^2-l_3^2)} ds = \tau.
\label{eq:res_int}
\end{equation}
For the cases $l^2+l_1^2-l_2^2-l_3^2 \neq 0$, we follow the construction from \cite{OS} and approximate the exponent by the dominant quadratic term, $2l_1^2$, which we can integrate exactly,
\begin{equation}
	\int^{\tau}_0 e^{2isl_1^2} ds = \tau \varphi_1(2i\tau l_1^2).
\end{equation}
Then, we approximate the integral as
\begin{align}
I_{{\rm cub}, 1}^{\tau}(v, t_n) \approx & \  I_{{\rm num}, 1}^{\tau}(v, t_n) \nn \\
: = & \  \sum_{l \in \mathbb{Z}}e^{ikx}\sum_{\substack{l_1, l_2, l_3 \in \mathbb{Z} \\ l = -l_1+l_2+l_3 \\ l^2 + l_1^2 -l_2^2 -l_3^2 \neq 0}}e^{i t_n(l^2+l_1^2-l_2^2-l_3^2)}\overline{ \hat{v}_{l_1}}  \hat{v}_{l_2}\hat{v}_{l_3}\int^{\tau}_0 e^{2isl_1^2} ds \nn\\
& \ +  \sum_{l \in \mathbb{Z}}e^{ilx}\sum_{\substack{l_1, l_2, l_3 \in \mathbb{Z} \\ l = -l_1+l_2+l_3 \\ l^2 + l_1^2 -l_2^2 -l_3^2 = 0}}\overline{ \hat{v}_{l_1}}  \hat{v}_{l_2}\hat{v}_{l_3}\tau \nn \\
= & \  \sum_{l \in \mathbb{Z}}e^{ilx}\sum_{\substack{l_1, l_2, l_3 \in \mathbb{Z}  \\ l = -l_1+l_2+l_3}}e^{i t_n(l^2+l_1^2-l_2^2-l_3^2)}\overline{ \hat{v}_{l_1}}  \hat{v}_{l_2}\hat{v}_{l_3}\tau \varphi_1(2i\tau l_1^2) \nn \\
& +  \sum_{l \in \mathbb{Z}}e^{ilx}\sum_{\substack{l_1, l_2, l_3 \in \mathbb{Z} \\ l = -l_1+l_2+l_3 \\ l^2 + l_1^2 -l_2^2 -l_3^2 = 0}}\overline{ \hat{v}_{l_1}}  \hat{v}_{l_2}\hat{v}_{l_3}\tau(1-\varphi_1(2i\tau l_1^2)).
\end{align}
According to the relationship
\begin{equation}
l^2 + l_1^2 - l_2^2 - l_3^2 = 2l_1^2 - 2l_1(l_2+l_3) + 2l_2l_3 = 2(l-l_2)(l-l_3),	
\end{equation}
we can write the second term as
\begin{align}
& \sum_{l \in \mathbb{Z}}e^{ilx}\sum_{\substack{l_1, l_2, l_3 \in \mathbb{Z} \\ l = -l_1+l_2+l_3 \\ l^2 + l_1^2 -l_2^2 -l_3^2 = 0}}\overline{ \hat{v}_{l_1}}  \hat{v}_{l_2}\hat{v}_{l_3}\tau(1-\varphi_1(2i\tau l_1^2)) \nn\\
 = & \ 2 \sum_{l \in \mathbb{Z}}e^{ilx}\sum_{\substack{l_1, l_2, l_3 \in \mathbb{Z} \\ l = -l_1+l_2+l_3 \\ l = l_2, l_1 = l_3}}\overline{ \hat{v}_{l_1}}  \hat{v}_{l_2}\hat{v}_{l_3}\tau(1-\varphi_1(2i\tau l_1^2))  -  \sum_{l \in \mathbb{Z}}e^{ilx}\sum_{\substack{l_1, l_2, l_3 \in \mathbb{Z} \\ l = -l_1+l_2+l_3 \\ l = l_2=l_1 = l_3}}\overline{ \hat{v}_{l_1}}  \hat{v}_{l_2}\hat{v}_{l_3}\tau(1-\varphi_1(2i\tau l_1^2)) \nn\\
 = & \ 2\tau (\widehat{g(v^n)})_0 v^n - \tau h(v^n),
\end{align}
with the auxiliary functions $a$ and $b$ given by 
\begin{equation}
g(u) = u(1-\varphi_1(-2i\tau\partial^2_x)) \overline{u}, \quad (\widehat{h(u)})_l = (1-\varphi_1(2il^2\tau)) \overline{\hat{u}_l}\hat{u}_l\hat{u}_l,
\label{eq:gh}
\end{equation}
which means that all of those terms can be computed in $\mathcal{O}(N \log N)$ operations. Compared with the scheme in \cite{OS}, the new scheme contains two additional terms, i.e.,
\begin{align}
v^{n+1} = \mathcal{S}_{\tau}(v^n) :=  v^n - i\tau \eps^2 \Bigg[&\ e^{-it_n\partial_x^2}\left((e^{it_n\partial_x^2}v^n)^2(\varphi_1(-2i\tau\partial_x^2)e^{-it_n\partial_x^2}\overline{v^n})\right)\nn\\
& \ +2(\widehat{g(v^n)})_0 v^n -  h(v^n)\Bigg],
\label{eq:MLI_v}
\end{align}
 with the initial data  $v^0 = \phi(x)$. Twisting the variable back, we obtain the non-resonant first-order low-regularity integrator (NRLI1) for $w^n = e^{it_n\partial_x^2}v^n$ as
\begin{align}
w^{n+1} = \widetilde{\mathcal{S}}_{\tau}(w^n) := &\ e^{i\tau \partial_x^2}\left[w^n - i\tau \eps^2 (w^n)^2(\varphi_1(-2i\tau\partial_x^2)\overline{w^n})\right]\nn\\
&\  -  2i\eps^2\tau (\widehat{g(w^n)})_0 e^{i\tau\partial_x^2}w^n + i\eps^2 \tau e^{i\tau\partial_x^2}h(w^n),
\label{eq:NRLI1}
\end{align}
with the initial data $w^0 = \phi(x)$ and $g$, $h$ defined in \eqref{eq:gh}.

\subsection{Improved uniform error bound of the NRLI1 \eqref{eq:NRLI1}}\label{sec:improved_bound_NRL1}
In this subsection, we employ the regularity compensation oscillation (RCO) technique to derive the long-time error bound for the method NRLI1 \eqref{eq:NRLI1}.

For the cubic NLSE \eqref{eq:CNLSE2}, we assume the exact solution $w(x, t)$ up to the time  $T_{\eps} = T/\eps^2$ with $T > 0$ fixed satisfies:
\[
{\rm(A)} \qquad
 \left\|w(x, t)\right\|_{L^{\infty}\left([0, T_{\varepsilon}]; H^{m}_{\rm per}\right)} \lesssim 1, \quad m \ge 2.
\]
Then we have the following error bound of the non-resonant low-regularity scheme \eqref{eq:NRLI1} for the cubic NLSE with $O(\eps^2)$-nonlinearity and $O(1)$-initial data up to the time of order $O(1/\eps^2)$.

\begin{theorem}
\label{thm:eb_nl}
Let $w^n$ be the numerical approximation obtained from the non-resonant first-order low-regularity (NRLI1) scheme \eqref{eq:NRLI1}. Under the assumption (A), for $0 < \tau_0 \leq 1$ sufficiently small and independent of $\eps$ such that, when $0< \tau \leq \frac{\alpha\pi \tau_0^2}{ (1+\tau_0)^2}< 1$ with a constant $\alpha\in(0,1)$, the following error bound holds
\begin{equation}
\left\|w(t_n) - w^n\right\|_1 \lesssim \eps^2\tau + \tau_0^{m-1}, \quad \left\|w^n\right\|_1 \leq 1 + M, \quad  0 \leq n \leq \frac{T/\eps^2}{\tau},
\label{eq:error_nl}
\end{equation}
where $M := \|w\|_{L^{\infty}([0, T_{\varepsilon}]; H^1)}$. In particular, if the exact solution is smooth, i.e.  $w(t) \in H^{\infty}_{\rm per}$,  the $\tau_0^{m-1}$ error part would decrease exponentially and can be ignored in practical computation when $\tau_0$ is small but fixed, and thus the estimate would practically become
\begin{equation}
\left\|w(t_n) - w^n\right\|_1 \lesssim \varepsilon^2\tau.
\end{equation}
\end{theorem}
\begin{remark}
The non-resonant low-regularity scheme admits approximations for more general initial data than classical time-splitting methods or exponential integrator methods. Compared with the proposed low-regularity scheme in \cite{OS}, it obtains the improved uniform error bound at $O(\eps^2\tau)$ up to the time of order $O(1/\eps^2)$ for the smooth initial data.
\end{remark}

As in the quadratic case, we just need to establish the error bound of the non-resonant low-regularity scheme for the twisted variable. This means we need to prove the following statement under the assumption that the exact solution $v(x, t)$ to \eqref{eq:CNLSE_v} up to the time $T_{\eps} = T/\eps^2$ with $T > 0$ fixed satisfies:
\[
{\rm(B)} \qquad
 \left\|v(x, t)\right\|_{L^{\infty}\left([0, T_{\varepsilon}]; H^{m}_{\rm per}\right)} \lesssim 1, \quad m \ge 2.
\]

\begin{theorem}
\label{thm:eb_mv}
Let $v^n$ be the numerical approximation obtained from the non-resonant low-regularity scheme \eqref{eq:MLI_v}. Under the assumption (B), for $0 < \tau_0 \leq 1$ sufficiently small and independent of $\eps$ such that, when $0< \tau \leq  \frac{\alpha\pi\tau_0^2}{ (1+\tau_0)^2}< 1$ with a constant $\alpha\in(0,1)$, the following error bound holds
\begin{equation}
\left\|v(t_n) - v^n\right\|_1 \lesssim \eps^2\tau + \tau_0^{m-1}, \quad \left\|v^n\right\|_1 \leq 1 + M, \quad  0 \leq n \leq \frac{T/\varepsilon^2}{\tau},
\label{eq:error_v}
\end{equation}
where $M := \|v\|_{L^{\infty}([0, T_{\varepsilon}]; H^1)}$. In particular, if the exact solution is smooth, i.e.  $v(x, t) \in H^{\infty}_{\rm per}$,  the $\tau_0^{m-1}$ error part would decrease exponentially and can be ignored in practical computation when $\tau_0$ is small but fixed, and thus the estimate would practically become
\begin{equation}
\left\|v(x, t_n) - v^n\right\|_1 \lesssim \varepsilon^2\tau.
\end{equation}
\end{theorem}

We begin with the following results for the local truncation error for the non-resonant first-order low-regularity scheme \eqref{eq:MLI_v}.
\begin{lemma}
 The local truncation error of the non-resonant first-order low-regularity scheme \eqref{eq:MLI_v} for the cubic NLSE \eqref{eq:CNLSE_v} can be written as $(0 \leq n \leq \frac{T/\eps^2}{\tau}-1)$
\begin{equation}\label{eq:local-nonlinear}
\mathcal{E}^{n} := \mathcal{S}_{\tau}(v(t_n)) - v(t_{n+1}) = \mathcal{R}(v(t_n)) + \mathcal{W}^n
\end{equation}
with
\begin{equation}
 \mathcal{R}(v(t_n)) := i\eps^2\left( I_{{\rm num},1}^{\tau}(v(t_n)) - \mathcal{I}_{{\rm cub}}^{\tau}(v, t_n)\right).
\end{equation}
Under the assumption (B), for $0 < \eps \leq 1$, we have the error bound
\begin{equation}
\left\| \mathcal{R}(v(t_n))\right\|_1 \lesssim \eps^2 \tau^2\|v(t_n)\|^3_2, \quad \left\| \mathcal{W}^n\right\|_1 \lesssim \eps^4 \tau^2.
\label{eq:local_nl}
\end{equation}
\label{lemma_nl}
\end{lemma}
\begin{proof} By the mild solution \eqref{eq:Duh_2}, the local truncation error can be written as
\begin{align}
\mathcal{E}^n = &\ \mathcal{S}_{\tau}(v(t_n)) + \mathcal{W}^n \nn\\
&\ - \left(v(t_n) -i\eps^2\int^{\tau}_0 e^{-i(t_n+s)\partial^2_x}\left[|e^{i(t_n+s)\partial^2_x}v(t_n)|^2 e^{i(t_n+s)\partial^2_x}v(t_n)\right]ds\right),	
\end{align}
with  $\mathcal{W}^n$ defined as
\begin{align}
 \mathcal{W}^n	=& \  i\eps^2\int^{\tau}_0 e^{-i(t_n+s)\partial^2_x}\left[|e^{i(t_n+s)\partial^2_x}v(t_n+s)|^2 e^{i(t_n+s)\partial^2_x}v(t_n+s)\right]ds \nn\\
 & \ - i\eps^2\int^{\tau}_0 e^{-i(t_n+s)\partial^2_x}\left[|e^{i(t_n+s)\partial^2_x}v(t_n)|^2 e^{i(t_n+s)\partial^2_x}v(t_n)\right]ds.
\end{align}
As $e^{it\Delta}$ is a linear isometry on $H^r$ for all $t \in \mathbb{R}$, we have
\begin{equation}
\left\|v(t_n + s) - v(t_n)\right\|_1 \leq \eps^2\int^s_0 \left\|v(t_n + \xi)\right\|_1^3 d\xi\leq \eps^2s \sup_{0\leq \xi\leq s}\left\|v(t_n + \xi)\right\|_1^3,
\end{equation}
which implies 
\begin{equation}
\left\| \mathcal{W}^n\right\|_1 \lesssim \eps^4\tau^2.	
\end{equation}
In the resonant case $l^2+l_1^2-l_2^2-l_3^2 = 0$, the integral $\int^{\tau}_0 e^{is(l^2+l_1^2-l_2^2-l_3^2)} ds$ is computed exactly. So we just need to consider the non-resonant cases $l^2+l_1^2-l_2^2-l_3^2 \neq 0$ as
\begin{align}
\mathcal{E}^{n} &=  \mathcal{R}(v(t_n))	+ \mathcal{W}^n\nn\\
& =\mathcal{W}^n+ i\eps^2 \sum_{l \in \mathbb{Z}}e^{ikx}\sum_{\substack{l_1, l_2, l_3 \in \mathbb{Z} \\ l = -l_1+l_2+l_3 \\ l^2 + l_1^2 -l_2^2 -l_3^2 \neq 0}} e^{i t_n(l^2+l_1^2-l_2^2-l_3^2)}\overline{ \hat{v}_{l_1}}  \hat{v}_{l_2}\hat{v}_{l_3}\nn\\
&\qquad \qquad\qquad\qquad\qquad \qquad \qquad \quad \int^{\tau}_0 \left(e^{2isl_1^2} -e^{is (l^2+l_1^2-l_2^2-l_3^2)} \right)ds,
\end{align}
with 
\begin{align}
&\left\|\mathcal{R}(v(t_n))\right\|^2_1 \nn \\
& = \eps^4 \sum_{l\in \mathbb{Z}}\left(1+|l|\right)^2\Bigg|\sum_{l \in \mathbb{Z}}e^{ikx}\sum_{\substack{l_1, l_2, l_3 \in \mathbb{Z} \\ l = -l_1+l_2+l_3 \\ l^2 + l_1^2 -l_2^2 -l_3^2 \neq 0}}e^{i t_n(l^2+l_1^2-l^2_2-l^2_3)}|2l_1(l_2+l_3)-2l_2l_3| \nn\\
& \qquad \qquad \qquad \qquad \quad \overline{\hat{v}_{l_1}}\hat{v}_{l_2}\hat{v}_{l_3} e^{ilx}\int^{\tau}_0 e^{2isl_1^2}\left(\frac{e^{is(-2l_1(l_2+l_3)-2l_2l_3)}-1}{s|2l_1(l_2+l_3)-2l_2l_3|}\right)s ds\Bigg|^2.
\end{align}
Since $\frac{e^{i\beta}-1}{|\beta|}$ is uniformly bounded for $\beta \in \mathbb{R}$, we obtain
\begin{align*}
&\left\|\mathcal{R}(v(t_n))\right\|^2_1 \nn\\
&\ \lesssim \eps^4 \tau^4\sum_{l \in \mathbb{Z}}	\left(1+|l|\right)^2 \Bigg(\sum_{l \in \mathbb{Z}}e^{ikx}\sum_{\substack{l_1, l_2, l_3 \in \mathbb{Z} \\ l = -l_1+l_2+l_3 \\ l^2 + l_1^2 -l_2^2 -l_3^2 \neq 0}}\left(|l_1l_2|+|l_1l_3|+|l_2l_3|\right)\hat{v}_{l_1}\hat{v}_{l_2}\hat{v}_{l_3}\Bigg)^2 \nn\\
&\ \lesssim\eps^4  \tau^4 \sum_{l\in \mathbb{Z}}	\left(1+|l|\right)^2 \Bigg(\sum_{l \in \mathbb{Z}}e^{ikx}\sum_{\substack{l_1, l_2, l_3 \in \mathbb{Z} \\ l = -l_1+l_2+l_3 \\ l^2 + l_1^2 -l_2^2 -l_3^2 \neq 0}}\left(1+|l_1|\right)\left(1+|l_2|\right)\left(1+|l_3|\right)\hat{v}_{l_1}\hat{v}_{l_2}\hat{v}_{l_3}\Bigg)^2.
\end{align*}
We introduce the auxiliary function $g(x) = \sum_{l \in \mathbb{Z}}\hat{g}_l e^{ilx}$ through its Fourier coefficients
\begin{equation}
\hat{g}_k = (1+|l|)|\hat{v}_l|.
\end{equation}
With the bilinear estimate \eqref{eq:bi}, we have
\begin{equation}
\left\|\mathcal{R}(v(t_n))\right\|_1 \lesssim \eps^2 \tau^2 \left\|g^3\right\|_1,	
\end{equation}
which implies
\begin{equation}
\left\|\mathcal{R}(v(t_n))\right\|_1 \lesssim \eps^2 \tau^2\left\|v\right\|^3_2.
\end{equation}
\end{proof}

\noindent
\emph{Proof for Theorem \ref{thm:eb_nl}.} Introducing the error function $e^n := e^n(x)$  by
\begin{equation}
e^n:= v^n - v(t_n), \quad n = 0, 1, \ldots,
\end{equation}
we apply a standard induction argument for proving the improved uniform error bound \eqref{eq:error_v}. Since $v^0 = \phi(x)$, it is obvious for $n = 0$. Assuming the error bounds \eqref{eq:error_v} hold true for all $0 \leq n \leq q \leq \frac{T/\varepsilon^2}{\tau}-1$, we are going to prove the case $n = q + 1$. For $0 \leq n \leq q$, we have
\begin{equation}
e^{n+1} =  v^{n+1} - \mathcal{S}_{\tau}(v(t_n)) + \mathcal{E}^{n}= e^n +Z^n + \mathcal{E}^{n},
 \label{eq:etg_nl}
\end{equation}
where $Z^n$ is given by 
\begin{align*}
Z^n = &\ i\eps^2 \tau \Bigg[e^{-it_n\partial_x^2}\left(\left(e^{it_n\partial_x^2}v(t_n)\right)^2\left(\varphi_1(-2i\tau\partial_x^2)e^{-it_n\partial_x^2}\overline{v(t_n)}\right)\right) \\
& \quad\qquad   - e^{-it_n\partial_x^2}\left((e^{it_n\partial_x^2}v^n)^2(\varphi_1(-2i\tau\partial_x^2)e^{-it_n\partial_x^2}\overline{v^n})\right) \nn\\
&  \quad\qquad  +2(\widehat{g(v(t_n))})_0 v(t_n) -2(\widehat{g(v^n)})_0 v^n -  h(v(t_n)) + h(v^n)\Bigg],
\end{align*}
with the bound
\begin{equation}
\left\|Z^n\right\|_1 \lesssim \eps^2 \tau  \|e^n\|_1.
\end{equation}
From \eqref{eq:etg_nl},  we obtain for $0\leq n\leq q$,
\begin{equation} 
\label{eq:n}
e^{n+1} = e^0+\sum\limits_{k = 0}^n \left(Z^k + {\mathcal{E}}^{k}\right).
\end{equation}
Since $e^0  = 0$, we get for $0 \leq n \leq q$,
\begin{equation}
\|e^{n+1}\|_1 \lesssim  \eps^2\tau^2+\varepsilon^2\tau\sum_{k=0}^n\|e^k\|_1 +\left\| \sum\limits_{k=0}^n \mathcal{R}(v(t_k))\right\|_1.
\label{eq:final}
\end{equation}
 For the twisted variable $v(t)$, we have $\|\partial_t\phi\|_{L^\infty([0,T/\eps^2];H^m)}\lesssim \eps^2$ and
\begin{equation}
\left\|v(t_n) - v(t_{n-1})\right\|_m \lesssim \eps^2 \tau, \quad 1 \leq n \leq \frac{T/\eps^2}{\tau}.
\label{eq:nl_twist}
\end{equation}
Following the RCO technique \cite{BCF}, we choose the cut-off parameter $\tau_0\in(0,1)$ and the corresponding Fourier modes $N_0=2\lceil 1/\tau_0\rceil$. Combining with \eqref{eq:final}, we have
\begin{equation}
\left\|e^{n+1}\right\|_1 \lesssim  \tau_0^{m-1} +\eps^2\tau^2+\eps^2\tau\sum_{k=0}^n\left\|e^k\right\|_1 +\|\mathcal{L}^n\|_1, 
\label{eq:final2}
\end{equation}
with
\begin{equation}
\mathcal{L}^n = \sum\limits_{k=0}^n P_{N_0}(\mathcal{R}(v(t_k))).	
\end{equation}

Define the index set
\begin{equation}
\mathcal{T}_{N_0} = \{l~|~l = -\frac{N_0}{2}, \ldots, \frac{N_0}{2}-1\},	
\end{equation}
and for $l\in\mathcal{T}_{N_0}$, define the index set $\mathcal{I}_l^{N_0}$ associated to $l$ as
\begin{equation}
\mathcal{I}_l^{N_0}=\left\{(l_1,l_2,l_3)\ \vert \ -l_1+l_2+l_3=l,\ l_1,l_2,l_3\in\mathcal{T}_{N_0}\right\}.
\end{equation}
Then, the expansion below follows
\begin{equation*}
P_{N_0}(\mathcal{R}(v(t_k))) =\sum\limits_{l\in\mathcal{T}_{N_0}}\sum\limits_{(l_1,l_2,l_3)\in\mathcal{I}_l^{N_0}}
\mathcal{G}_{k,l,l_1,l_2,l_3}(s)e^{ilx},
\end{equation*}
where the coefficients $\mathcal{G}_{k,l,l_1,l_2,l_3}(s)$ are functions of $s$ only,
\begin{equation}
\mathcal{G}_{k,l,l_1,l_2,l_3}(s) = e^{i(t_k+s)\delta_{l,l_1,l_2,l_3}}\left(\hat{v}_{l_1}(t_k)\right)^{\ast}\hat{v}_{l_2}(t_k)\hat{v}_{l_3}(t_k),
\label{eq:mGdeff}
\end{equation}
and $\delta_{l,l_1,l_2,l_3} = l^2 + l_1^2 - l_2^2  - l_3^2$.
The remainder term in \eqref{eq:final} reads
\begin{align}\label{eq:remainder-dec}
\mathcal{L}^n(x)  =  i\eps^2\sum\limits_{k=0}^n
\sum\limits_{l\in\mathcal{T}_{N_0}}\sum\limits_{(l_1,l_2,l_3)\in\mathcal{I}_l^{N_0}}\Lambda_{k,l,l_1,l_2,l_3} e^{ilx},
\end{align}
where $\Lambda_{k,l,l_1,l_2,l_3} = 0$ for $\delta_{l,l_1,l_2,l_3} = 0$ and for $\delta_{l,l_1,l_2,l_3} \neq 0$, we have
\begin{align}
\Lambda_{k,l,l_1,l_2,l_3} &= \int_0^\tau\mathcal{G}_{k,l,l_1,l_2,l_3}( s)\,d s - \int_0^\tau e^{2i(t_k+s)l_1^2} \left(\hat{v}_{l_1}(t_k)\right)^{\ast}\hat{v}_{l_2}(t_k)\hat{v}_{l_3}(t_k) \,d s\nn\\
 &= r_{l,l_1,l_2,l_3}e^{it_k\delta_{l,l_1,l_2,l_3}}c_{k,l,l_1,l_2,l_3},
\label{eq:Lambdak}
\end{align}
with coefficients  $c_{k,l,l_1,l_2,l_3}$ and $r_{l,l_1,l_2,l_3}$ given by
\begin{align}
c_{k,l,l_1,l_2,l_3}=& \ (\hat{v}_{l_1}(t_k))^{\ast}\hat{v}_{l_2}(t_k)\hat{v}_{l_3}(t_k),\label{eq:calFdef}\\
r_{l,l_1,l_2,l_3}= & \ \int_0^\tau e^{is\delta_{l,l_1,l_2,l_3}}\,d s -  \int_0^\tau e^{isl_1^2}\,d s =  O\left(\tau^2 \vert\delta_{l,l_1,l_2,l_3}- 2l_1^2\vert\right).\label{eq:rest}
\end{align}
Since $\Lambda_{k,l,l_1,l_2,l_3} = 0$ for $\delta_{l,l_1,l_2,l_3} = 0$, we only need consider the case $\delta_{l,l_1,l_2,l_3}\neq0$. First, for $l\in\mathcal{T}_{N_0}$ and $(l_1,l_2,l_3)\in\mathcal{I}_l^{N_0}$, we have
\begin{equation}
\vert\delta_{l,l_1,l_2,l_3}\vert \leq 2\delta_{N_0/2}= 2\left(\frac{N_0}{2}\right)^2 \leq \frac{2(1+\tau_0)^2}{\tau_0^2},
\end{equation}
which implies for $0< \tau \leq \frac{\alpha \pi\tau_0^2}{(1+\tau_0)^2}$ with $0<\tau_0, \alpha<1$,
\begin{equation}\label{eq:cfl}
\frac{\tau}{2}\vert \delta_{l,l_1,l_2,l_3} \vert \leq \alpha\pi.
\end{equation}
Denoting $S_{n,l,l_1,l_2,l_3}=\sum_{k=0}^ne^{it_k\delta_{l,l_1,l_2,l_3}}$ ($n\ge0$) and using summation-by-parts formula, we find from \eqref{eq:Lambdak} that
\begin{align}
\sum_{k=0}^n\Lambda_{k,l,l_1,l_2,l_3} =& \
r_{l,l_1,l_2,l_3}\sum_{k=0}^{n-1}S_{k,l,l_1,l_2,l_3} \left(c_{k,l,l_1,l_2,l_3}-c_{k+1,l,l_1,l_2,l_3}\right)\nonumber\\
&\ +S_{n,l,l_1,l_2,l_3}\, r_{l,l_1,l_2,l_3}\,c_{n,l,l_1,l_2,l_3},\label{eq:lambdasum}
\end{align}
and
\begin{align}
&c_{k,l,l_1,l_2,l_3}-c_{k+1,l,l_1,l_2,l_3}\nn\\
& = (\widehat{v}_{l_1}(t_k))^{\ast}(\widehat{v}_{l_2}(t_k)-\widehat{v}_{l_2}(t_{k+1})) \widehat{v}_{l_3}(t_k) + (\widehat{v}_{l_1}(t_k)-\widehat{v}_{l_1}(t_{k+1}))^{\ast}\widehat{v}_{l_2}(t_{k+1})\widehat{v}_{l_3}(t_k)\nn\\
&\;\;\;\;\; + (\widehat{v}_{l_1}(t_{k+1}))^{\ast}\widehat{v}_{l_2}(t_{k+1}) (\widehat{v}_{l_3}(t_k)-\widehat{v}_{l_3}(t_{k+1})),\label{eq:cksum}
\end{align}
where $c^\ast$ is the complex conjugate of $c$. We know from \eqref{eq:cfl} that
for  $C=\frac{2\alpha}{\sin(\alpha\pi)}$,
\begin{equation}\label{eq:Sbd}
\vert S_{n,l,l_1,l_2,l_3}\vert \leq \frac{1}{\vert\sin(\tau \delta_{l,l_1,l_2,l_3}/2)\vert}\leq\frac{C}{\tau\vert\delta_{l,l_1,l_2,l_3}\vert},\quad \forall n\ge0.
\end{equation}
Combining \eqref{eq:rest}, \eqref{eq:lambdasum}, \eqref{eq:cksum} and \eqref{eq:Sbd}, we have
\begin{align}
\left\vert\sum_{k=0}^n\Lambda_{k,l,l_1,l_2,l_3}\right\vert
\lesssim & \  \tau \frac{\vert \delta_{l,l_1,l_2,l_3} - 2l_1^2 \vert}{\vert \delta_{l,l_1,l_2,l_3} \vert}  \sum\limits_{k=0}^{n-1}\bigg(
\left\vert\hat{v}_{l_1}(t_k)-\hat{v}_{l_1}(t_{k+1})\right\vert\left\vert\hat{v}_{l_2}(t_k)\right\vert \left\vert\hat{v}_{l_3}(t_k)\right\vert\nn\\
&\ +\left\vert\hat{v}_{l_1}(t_{k+1})\right\vert\left\vert\hat{v}_{l_2}(t_k)-\hat{v}_{l_2}(t_{k+1})\right\vert \left\vert\hat{v}_{l_3}(t_k)\right\vert \nn\\
&\ +\left\vert\hat{v}_{l_1}(t_{k+1})\right\vert\left\vert\hat{v}_{l_2}(t_{k+1})\right\vert \left\vert\hat{v}_{l_3}(t_k)-\hat{v}_{l_3}(t_{k+1})\right\vert\bigg)\nonumber\\
&\ + \tau  \frac{\vert \delta_{l,l_1,l_2,l_3} - 2l_1^2 \vert}{\vert \delta_{l,l_1,l_2,l_3} \vert}\left\vert\hat{v}_{l_1}(t_n)\right\vert\left\vert\hat{v}_{l_2}(t_n)\right\vert \left\vert\hat{v}_{l_3}(t_n)\right\vert.
\label{eq:sumlambda}
\end{align}

Since $\delta_{l,l_1,l_2,l_3}- 2l_1^2= - 2l_1(l_2+l_3)+2l_2l_3$, we have for $l\in\mathcal{T}_{N_0}$ and $(l_1,l_2,l_3)\in\mathcal{I}_l^{N_0}$, there holds
\begin{equation}
(1+\vert l\vert) \frac{\vert \delta_{l,l_1,l_2,l_3}- 2l_1^2 \vert}{\vert \delta_{l,l_1,l_2,l_3}\vert} \lesssim \prod_{j = 1}^3 (1+ \vert l_j\vert).
\label{eq:mlbd}
\end{equation}
Based on \eqref{eq:remainder-dec}, \eqref{eq:sumlambda} and \eqref{eq:mlbd},  we have from \eqref{eq:final},
\begin{align}
&\|\mathcal{L}^n \|^2_1  \nn \\
& = \ \eps^4
\sum\limits_{l\in\mathcal{T}_{N_0}}\left(1+l^2\right) \big\vert\sum\limits_{(l_1,l_2,l_3)\in\mathcal{I}_l^{N_0}}\sum\limits_{k=0}^n\Lambda_{k,l,l_1,l_2,l_3}\big\vert^2 \nn \\
& \lesssim \ \eps^4\tau^2
\bigg\{\sum_{l\in\mathcal{T}_{N_0}}\bigg(\sum\limits_{(l_1,l_2,l_3)\in\mathcal{I}_l^{N_0}}\left\vert\hat{v}_{l_1}(t_n)\right\vert\left\vert\hat{v}_{l_2}(t_n)\right\vert \left\vert\hat{v}_{l_3}(t_n)\right\vert\prod_{j=1}^3(1+  \vert l_j\vert) \bigg)^2
\nn \\
&\quad+n \sum\limits_{k=0}^{n-1}
\sum_{l\in\mathcal{T}_{N_0}}\bigg[\bigg(\sum\limits_{(l_1,l_2,l_3)\in\mathcal{I}_l^{N_0}}
\left\vert\hat{v}_{l_1}(t_k)-\hat{v}_{l_1}(t_{k+1})\right\vert\left\vert\hat{v}_{l_2}(t_k)\right\vert \left\vert\hat{v}_{l_3}(t_k)\right\vert \prod_{j=1}^3(1+ \vert l_j\vert) \bigg)^2\nn\\
& \;\;\;\; +\bigg(\sum\limits_{(l_1,l_2,l_3)\in\mathcal{I}_l^{N_0}}
\left\vert\hat{v}_{l_1}(t_{k+1})\right\vert\left\vert\hat{v}_{l_2}(t_k)-\hat{v}_{l_2}(t_{k+1})\right\vert \left\vert \hat{v}_{l_3}(t_k)\right\vert\prod_{j=1}^3(1+ \vert l_j\vert) \bigg)^2\nn\\
& \;\;\;\; +\bigg(\sum\limits_{(l_1,l_2,l_3)\in\mathcal{I}_l^{N_0}}
\left\vert\hat{v}_{l_1}(t_{k+1})\right\vert\left\vert\hat{v}_{l_2}(t_{k+1})\right\vert \left\vert\hat{v}_{l_3}(t_k)-\hat{v}_{l_3}(t_{k+1})\right\vert\prod_{j=1}^3(1+ \vert l_j\vert) \bigg)^2\bigg]\bigg\}.
\label{eq:sumlambda-2}
\end{align}
Introducing the auxiliary function $\xi(x)=\sum_{l\in\mathbb{Z}}(1+\vert l \vert)\left\vert\hat{v}_l(t_n)\right\vert e^{ilx}$, where $\xi(x)\in H_{\rm per}^{m-1}(\Omega)$ implied by assumption (B) and $\|\xi\|_{H^s}\lesssim \|v(t_n)\|_{H^{s+1}}$. Expanding 
\begin{equation}
\vert\xi(x)\vert^2\xi(x)=\sum\limits_{l\in\mathbb{Z}}\sum\limits_{l=-l_1+l_2+l_3} \prod_{j=1}^3\left((1+\vert l_j\vert)\left\vert\hat{v}_{l_j}(t_n)\right\vert\right) e^{ilx},
\end{equation}
we get
\begin{align}
&\sum_{l\in\mathcal{T}_{N_0}} \bigg(\sum\limits_{(l_1,l_2,l_3)\in\mathcal{I}_l^{N_0}}\left\vert\hat{v}_{l_1}(t_n)\right\vert\left\vert\hat{v}_{l_2}(t_n)\right\vert \left\vert\hat{v}_{l_3}(t_n)\right\vert\prod_{j=1}^3(1+\vert l_j\vert)\bigg)^2\nn\\
& \leq \left\|\vert\xi(x)\vert^2\xi(x)\right\|^2 \lesssim \left\|\xi(x)\right\|_1^6\lesssim \left\|v(t_k)\right\|_2^6 \lesssim1.
\end{align}
Noticing \eqref{eq:nl_twist}, we can estimate each terms in \eqref{eq:sumlambda-2} accordingly as
\begin{align}
&\left\| \sum\limits_{k=0}^n P_{N_0}\mathcal{R}(v(t_k))\right\|^2_1 \nn \\
& \lesssim \eps^4\tau^2 \bigg[\left\|v(t_k)\right\|_2^6+n\sum\limits_{k=0}^{n-1}
\left\|v(t_k) -v(t_{k+1})\right\|_2^2(\left\|v(t_k)\right\|_2 + \left\|v(t_{k+1})\right\|_2)^4\bigg]\nn \\
&\lesssim  \eps^4\tau^2+n^2\eps^4\tau^2 (\eps^2\tau)^2
\lesssim \eps^4\tau^2,\quad n\leq q,
\label{eq:est-l2}
\end{align}
and \eqref{eq:final2} implies
\begin{equation}
\|e^{n+1}\|_1 \lesssim \tau_0^{m-1} +\eps^2\tau+\eps^2\tau\sum_{k=0}^n\|e^k\|_1,\quad 0\leq n\leq q. \label{eq:final3}
\end{equation}
Using discrete Gronwall's inequality, we have
\begin{equation}
\|e^{q+1}\|_1\lesssim \eps^2\tau + \tau_0^{m-1},\quad 0\leq q\leq\frac{T/\eps^2}{\tau}-1,
\end{equation}
which implies the first inequality in \eqref{eq:error_nl} at $n = q+1$. For  $0 < \tau_0 \leq 1$, when $0< \tau \leq \frac{\alpha \pi \tau_0^2}{ (1+\tau_0)^2}< 1$, the triangle inequality yields that
\begin{equation*}
\left\|v^{q+1}\right\|_1 \leq \left\|v(x, t_{q+1})\right\|_1 + \left\|e^{q+1}\right\|_1 \leq M + 1, \quad 0 \leq q \leq \frac{T/\varepsilon^2}{\tau}-1,	
\end{equation*}
which means that the induction process for \eqref{eq:error_nl} is completed.  $\hfill\Box$

\subsection{A non-resonant symmetric second-order low-regularity integrator}
In this subsection, we will design a non-resonant symmetric second-order low-regularity integrator which can deal with rough initial data and obtain the improved uniform error bound for smooth initial data as well. For the symmetric scheme, we require the map is self-adjoint \cite{BMS,CCO1}, i.e., $\widetilde{\mathcal{S}}_{\tau}=\widetilde{\mathcal{S}}_{-\tau}^{-1}$. Following the idea in \cite{HLW} we construct our symmetric second order method as a composition scheme, in the form
\begin{equation}
w^{n+1} = \widetilde{\mathcal{S}}_{-\tau/2}^{-1} \circ\widetilde{\mathcal{S}}_{\tau/2}(w^n).
\end{equation}
Recalling the first-order scheme
\begin{align}
w^{n+1} = \widetilde{\mathcal{S}}_{\tau}(w^n) := &\ e^{i\tau \partial_x^2}\left[w^n - i\tau \eps^2 (w^n)^2(\varphi_1(-2i\tau\partial_x^2)\overline{w^n})\right]\nn\\
&\  -  2i\eps^2\tau (\widehat{g(w^n)})_0 e^{i\tau\partial_x^2}w^n + i\eps^2 \tau e^{i\tau\partial_x^2}h(w^n),
\end{align}
we have the following non-resonant symmetric second-order low-regularity scheme (NRSLI2) 
\begin{align}
\label{eq:msu}
w^{n+1} =&\  \widetilde{\mathcal{M}}_{\tau}(w^n, w^{n+1}) \nn\\
:= & \  e^{i\tau \partial_x^2}\Bigg[w^n - \frac{i\tau\eps^2}{2}(w^n)^2(\varphi_1(-i\tau\partial_x^2)\overline{w^n})\Bigg]  -   \frac{i\eps^2\tau}{2}\left((w^{n+1})^2(\varphi_1(i\tau\partial_x^2)\overline{w^{n+1}})\right) \nn\\
& -  \frac{i\eps^2\tau}{2} \Bigg[2(\widehat{g(w^n)})_0 e^{i\tau\partial_x^2}w^n - e^{i\tau\partial_x^2}h(w^n) +  2 (\widehat{g(w^{n+1})})_0 w^{n+1} - h(w^{n+1})\Bigg],
\end{align}
with the initial data $w^0 = \phi(x)$ and $g$, $h$ defined in \eqref{eq:gh}.

\subsection{Improved uniform error bound of the NRSLI2 \eqref{eq:msu}}
In this subsection, we carry out the long-time error bound for the non-resonant symmetric second-order low-regularity (NRSLI2) scheme \eqref{eq:msu}. 

For the cubic NLSE \eqref{eq:CNLSE2}, we assume the exact solution $w(x, t)$ up to the time  $T_{\eps} = T/\eps^2$ with $T > 0$ fixed satisfies:
\[
{\rm(C)} \qquad
 \left\|w(x, t)\right\|_{L^{\infty}\left([0, T_{\varepsilon}]; H^{m}_{\rm per}\right)} \lesssim 1, \quad m \ge 4.
\]
Then we have the following improved uniform error bound of the NRSLI2 \eqref{eq:msu} for the cubic NLSE with $O(\eps^2)$-nonlinearity up to the times of order $O(1/\eps^2)$.

\begin{theorem}
\label{thm:eb_w2}
Let $w^n$ be the numerical approximation obtained from the non-resonant symmetric second-order low-regularity (NRSLI2) scheme \eqref{eq:msu}. Under the assumption (C),  for $0 < \tau_0 \leq 1$ sufficiently small and independent of $\eps$ such that, when $0< \tau \leq \frac{\alpha \pi \tau_0^2}{ (1+\tau_0)^2}< 1$ with a constant $\alpha\in(0,1)$, the following error bound holds
\begin{equation}
\left\|w(t_n) - w^n\right\|_1 \lesssim \eps^2\tau^2 + \tau_0^{m-1}, \quad \left\|w^n\right\|_1 \leq 1 + M, \quad  0 \leq n \leq \frac{T/\eps^2}{\tau},
\label{eq:error_nl2}
\end{equation}
where $M := \|w\|_{L^{\infty}([0, T_{\eps}]; H^1)}$. In particular, if the exact solution is smooth, i.e.  $w(t) \in H^{\infty}_{\rm per}$,  the $\tau_0^{m-1}$ error part would decrease exponentially and can be ignored in practical computation when $\tau_0$ is small but fixed, and thus the estimate would practically become
\begin{equation}
\left\|w(t_n) - w^n\right\|_1 \lesssim \eps^2\tau^2.
\end{equation}
\end{theorem}
\begin{remark}
The second-order time-splitting method requires four additional spatial derivatives to get the improved uniform error bounds, while the proposed low-regularity scheme \eqref{eq:msu} requires three additional spatial derivatives. Compared with the classical low-regularity scheme in \cite{OS}, the new non-resonant scheme obtains the improved uniform error bound at $O(\eps^2\tau^2)$ up to the time of order $O(1/\eps^2)$ for the smooth initial data.
\end{remark}

As in the proof of Theorem \ref{thm:eb_nl}, we begin with the error estimates for the local truncation error in the following lemma. The proof proceeds analogously to the estimates in \cite{BYA, BMS} and we omit the details here for brevity.
\begin{lemma}
 The local truncation error of the non-resonant symmetric second-order scheme \eqref{eq:msu} for the cubic NLSE \eqref{eq:CNLSE2} can be written as $(0 \leq n \leq \frac{T/\eps^2}{\tau}-1)$
\begin{equation}
\label{eq:local-nl}
\widetilde{\mathcal{E}}^{n} := \widetilde{\mathcal{M}}_{\tau}(w(t_n), w(t_{n+1})) - w(t_{n+1}) = \widetilde{\mathcal{R}}(w(t_n), w(t_{n+1})) + \widetilde{\mathcal{W}}^n,
\end{equation}
then under the assumption (C), for $0 < \eps \leq 1$, we have the error bound
\begin{equation}
\left\|\widetilde{\mathcal{R}}(w(t_n), w(t_{n+1}))\right\|_1 \lesssim \eps^2 \tau^3\left(\|w(t_n)\|^3_4 +\|w(t_{n+1})\|^3_4\right), \quad \left\| \widetilde{\mathcal{W}}^n\right\|_1\lesssim \eps^4 \tau^3.
\end{equation}
\end{lemma}

\noindent
\emph{Proof for Theorem \ref{thm:eb_w2}.} For simplicity, we only show the difference of the proof between the first-order and second-order scheme. Introducing the error function $\widetilde{e}^n := \widetilde{e}^n(x)$  by
\begin{equation}
\widetilde{e}^n:= w^n - w(t_n), \quad n = 0, 1, \ldots,
\end{equation}
we apply a standard induction argument for proving the improved uniform error bound \eqref{eq:error_nl2}. Since $w^0 = \phi(x)$, it is obvious for $n = 0$. Assuming the error bound \eqref{eq:error_nl2} holds true for all $0 \leq n \leq q \leq \frac{T/\varepsilon^2}{\tau}-1$, we are going to prove the case $n = q + 1$. For $0 \leq n \leq q$, with similar procedure for the first-order scheme, we have 
\begin{equation}
\left\|\widetilde{e}^{n+1}\right\|_1 \lesssim  \tau_0^{m-1} +\eps^2\tau^2+\eps^2\tau\sum_{k=0}^n\left\|\widetilde{e}^k\right\|_1 +\|\widetilde{\mathcal{L}}^n\|_1, 
\label{eq:final2}
\end{equation}
with
\begin{equation}
\widetilde{\mathcal{L}}^n= \sum\limits_{k=0}^n P_{N_0}(\widetilde{\mathcal{R}}(w(t_k), w(t_{k+1}), t_k)).	
\end{equation}
For the last term in the right hand side of \eqref{eq:final2}, we have 
\begin{align*}
&\left\| \sum\limits_{k=0}^n P_{N_0}\widetilde{\mathcal{R}}(w(t_k), w(t_{k+1}))\right\|^2_1 \nn \\
& \lesssim \eps^4\tau^4 \bigg[\left\|w(t_n)\right\|_2^6+\left\|w(t_{n+1})\right\|_2^6+n\sum\limits_{k=0}^{n-1}
\left\|w(t_k) -w(t_{k+1})\right\|_2^2(\left\|w(t_k)\right\|_2 + \left\|w(t_{k+1})\right\|_2)^4\bigg]\nn \\
&\lesssim  \eps^4\tau^4 + n^2\eps^4\tau^2 (\eps^2\tau^2)^2
\lesssim \eps^4\tau^4,\quad n\leq q,
\end{align*}
and \eqref{eq:final2} implies
\begin{equation}
\|\widetilde{e}^{n+1}\|_1 \lesssim \tau_0^{m-1} +\eps^2\tau+\eps^2\tau^2\sum_{k=0}^n\|\widetilde{e}^k\|_1,\quad 0\leq n\leq q.
\end{equation}
Using discrete Gronwall's inequality, we have
\begin{equation}
\|\widetilde{e}^{q+1}\|_1 \lesssim \eps^2\tau^2 + \tau_0^{m-1},\quad 0\leq q\leq\frac{T/\eps^2}{\tau}-1,
\end{equation}
which implies the first inequality in \eqref{eq:error_nl} at $n = q+1$. For $0 < \tau_0 \leq 1$, when $0< \tau \leq \frac{\alpha \tau_0^2}{ (1+\tau_0)^2}< 1$, the triangle inequality yields that
\begin{equation*}
\left\|w^{q+1}\right\|_1 \leq \left\|w(x, t_{q+1})\right\|_1 + \left\|\widetilde{e}^{q+1}\right\|_1 \leq M + 1, \quad 0 \leq q \leq \frac{T/\eps^2}{\tau}-1,	
\end{equation*}
which means that the induction process for \eqref{eq:error_nl2} is completed.  $\hfill\Box$

\section{Numerical results}\label{sec:numerical_results}
Having understood the proofs of our main results in Theorems \ref{thm:w1}, \ref{thm:w2}, \ref{thm:eb_nl} \& \ref{thm:eb_w2}, we can now turn to some numerical results comparing the theory with practical observations. {In the following numerical experiments,} we use a standard Fourier spectral method as the basis for our spatial discretisation. In all experiments, we use  $M=2^{11}$ Fourier modes and the reference solutions are computed with the symmetric schemes introduced above with a fine time stepsize $\tau=10^{-6}$. Throughout these numerical experiments we choose initial data of the following form (cf. \cite[Section~5.1]{OS}): First we fix the number of Fourier modes, $2M$, in the spatial discretisation and take a sample of a vector of uniformly random distributed complex numbers:
\begin{align*}
	\mathbf{U}=\left(U_{-M+1},\dots, U_{M}\right), \quad U_j\sim U([0,1+i]), j=1,\dots, 2M.
\end{align*}
Our initial condition is then $u(0,\cdot)$ given by its Fourier coefficients for a specified value of $\theta$:
\begin{align}\label{eqn:form_of_random_initial_data}
	\hat{u}^{0}_{-M+j}=\langle -M+j\rangle^{-\theta} U_{-M+j},\quad j=1,\dots, 2M,\, \text{where\ }\langle m\rangle=\begin{cases}|m|,&m\neq 0,\\
		1,& m=0.
	\end{cases}
\end{align}
This initial condition is almost surely (with respect to the joint probability measure of the uniform distributions) in $H^{\theta}$.
\subsection{Long-time behaviour for the quadratic NLSE}
Let us begin with comparing the results of Theorems \ref{thm:w1} \& \ref{thm:w2} with practical experiments. In Figure \ref{fig:quadratic_err_fn_of_epsilon}, we can observe the clear linear dependence on $\varepsilon$ of the error at $t=T/\varepsilon$ as predicted in Theorems \ref{thm:w1} \& \ref{thm:w2}. In addition, Figure~\ref{fig:quadratic_err_fn_of_t} shows the good long-time behaviour of the error in the resonance-based schemes (note the initial data here is of the form \eqref{eqn:form_of_random_initial_data} with $\theta=2.0$), which appears to be well-preserved even when the error in splitting methods increases. 

Furthermore, in Figure~\ref{fig:quadratic_dependence_on_tau} we can observe that at $t=T/\varepsilon$ we have the predicted linear/quadratic dependence on the time-step $\tau$.%\newpage

\begin{figure}[h!]
	\centering
		\centering
		\includegraphics[width=0.5\textwidth]{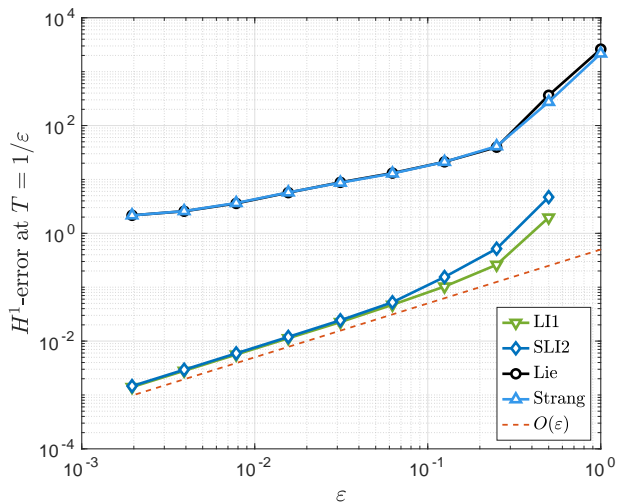}\vspace{-0.2cm}
	\caption{Long-time $H^1$-error for initial data in $H^1$ ($\theta=1$) with fixed $\tau=0.25$.}
	\label{fig:quadratic_err_fn_of_epsilon}
	%\label{fig:rational_end_times}
\end{figure}
\begin{figure}[h!]%\vspace{-1cm}
	\centering
	\includegraphics[width=0.5\textwidth]{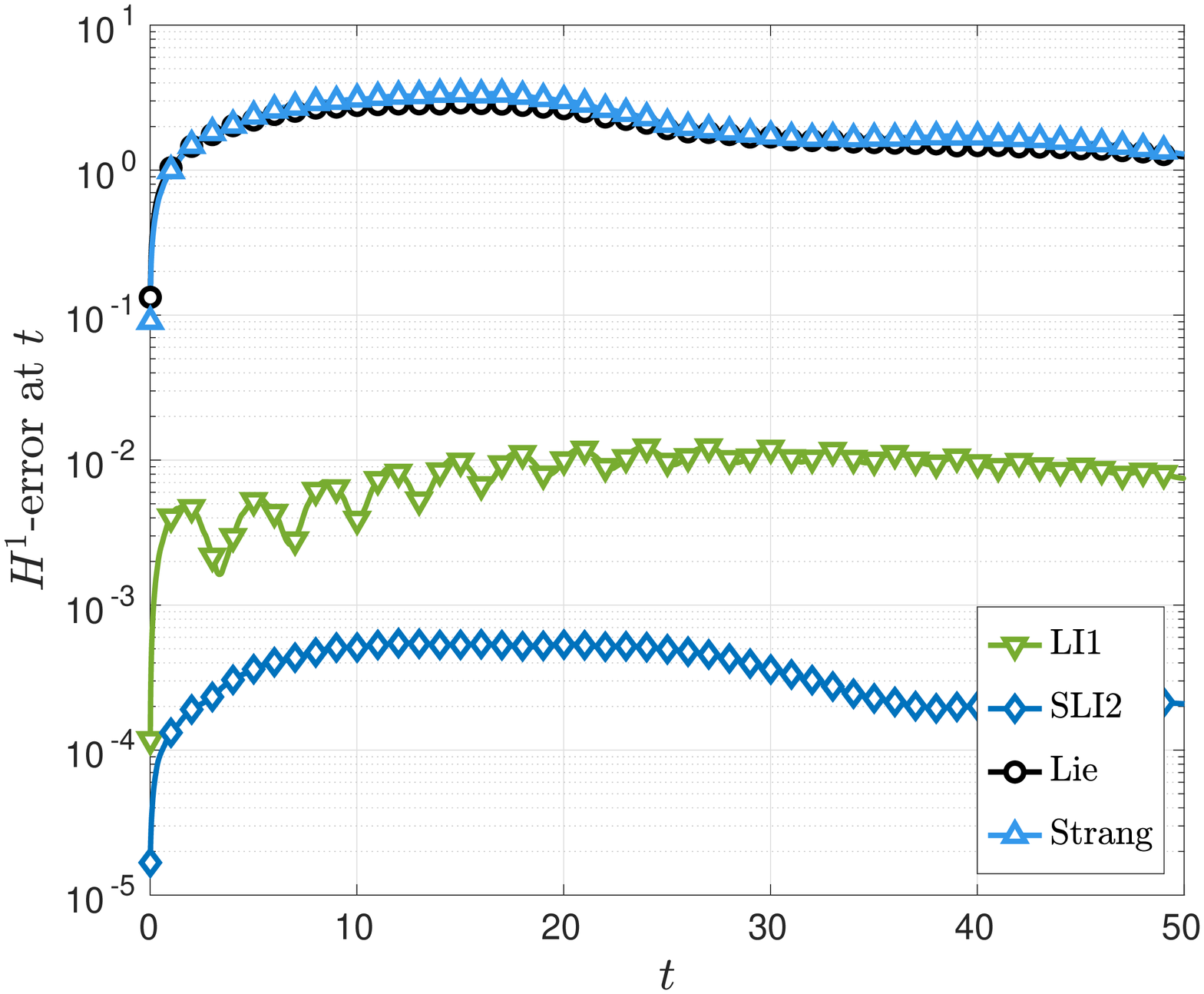}\vspace{-0.2cm}
	\caption{Error as a function of $t$ for initial data in $H^1$ ($\theta=1$) with fixed $\tau=0.01$.}
	\label{fig:quadratic_err_fn_of_t}
\end{figure}

\begin{figure}[h!]
	\centering
	\begin{subfigure}{0.495\textwidth}
		\centering
		\includegraphics[width=0.985\textwidth]{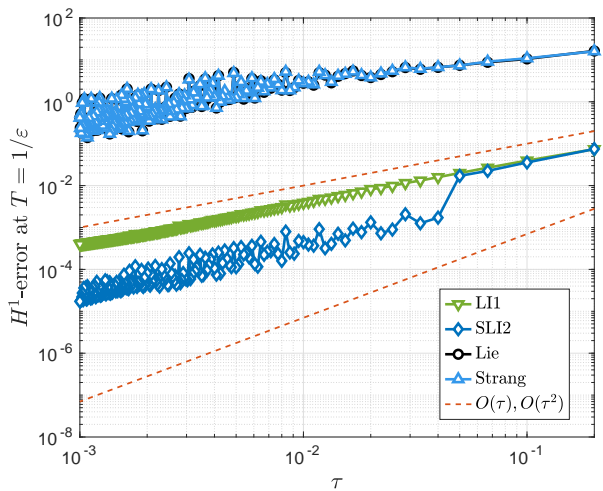}
		\caption{$H^1$ initial data ($\theta=1$).}
	\end{subfigure}

\begin{subfigure}{0.495\textwidth}
	\centering
	\includegraphics[width=0.985\textwidth]{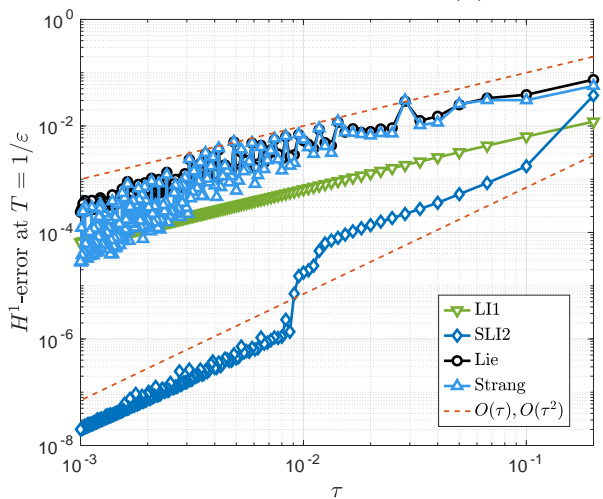}
	\caption{$H^2$ initial data ($\theta=2$).}
\end{subfigure}
	\begin{subfigure}{0.495\textwidth}
	\centering
	\includegraphics[width=0.985\textwidth]{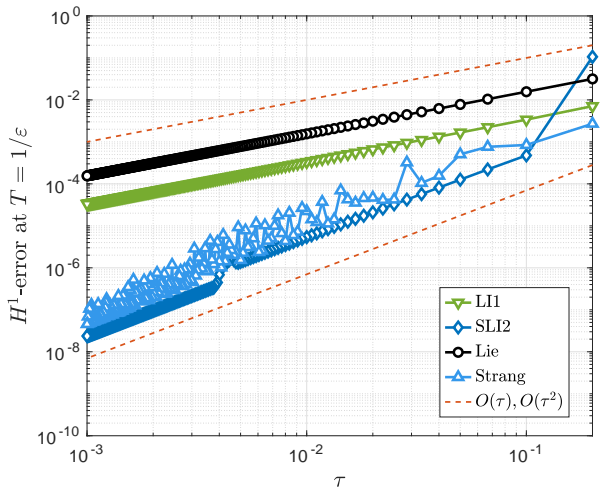}
	\caption{$H^3$ initial data ($\theta=3$).}
\end{subfigure}
	\caption{Long-time $H^1$-error as a function of $\tau$ with fixed $\varepsilon=0.1$.}
	\label{fig:quadratic_dependence_on_tau}
\end{figure}

\newpage\subsection{Long-time behaviour for the cubic NLSE}
We can perform similar numerical experiments for the case of the cubic NLSE. In the following we also include the reference solution ``Ostermann \& Schratz '18'' which is the method introduced in \cite{OS}. The central conclusion from our results in Theorems~\ref{thm:eb_nl} \& \ref{thm:eb_w2} is that our novel non-resonant low-regularity integrators have an error at $t=T/\varepsilon^2$ which decays quadratically in $\varepsilon^2$. In Figure~\ref{fig:cubic_err_fn_of_epsilon}, we observe the clear improvement over the standard low-regularity integrator first introduced in \cite{OS} confirming that the improved long-time behaviour is observed even at the level of $H^2$ initial data and thus underlying the significance of these novel schemes.
\begin{figure}[h!]
	\centering
	\includegraphics[width=0.5\textwidth]{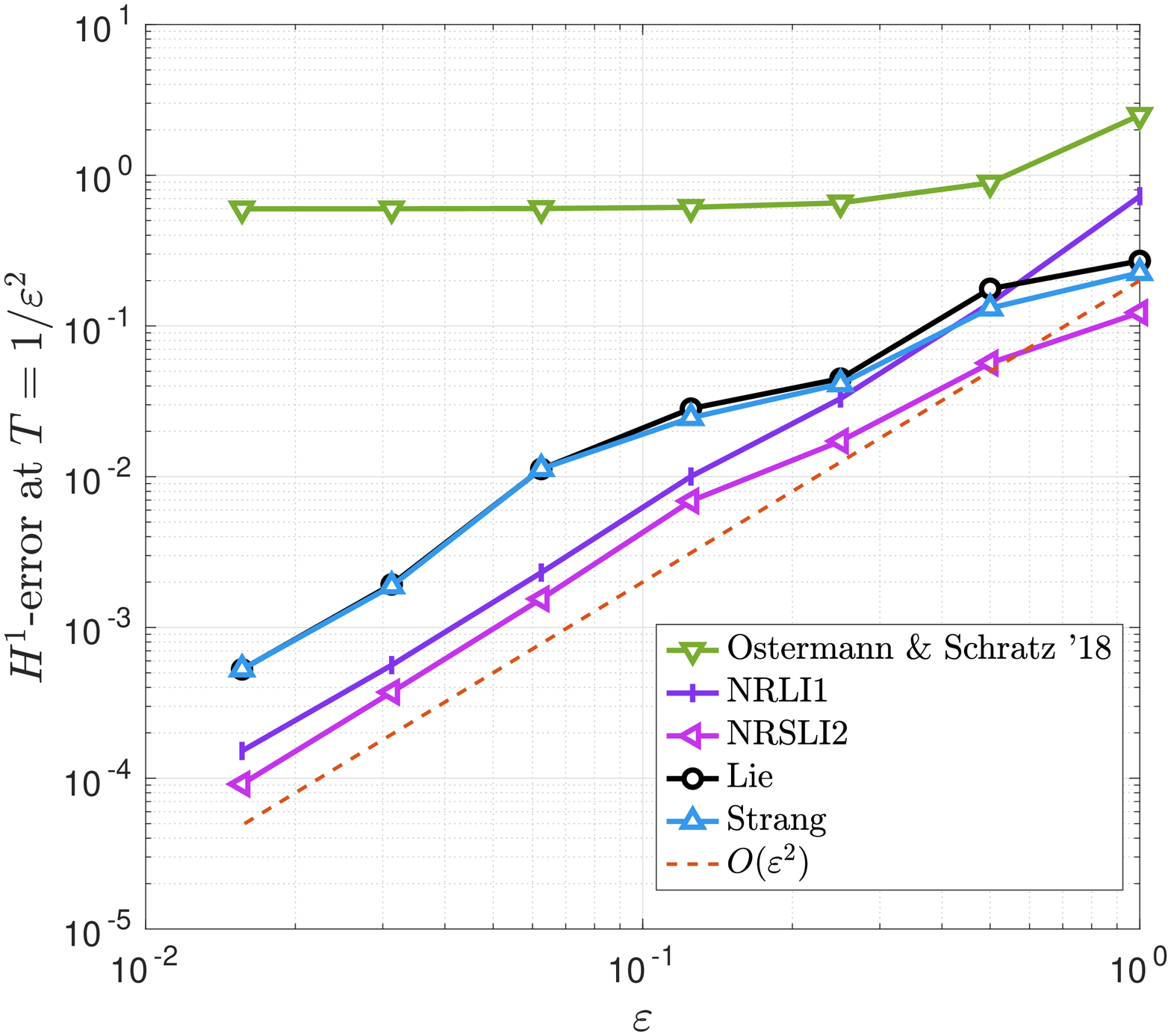}
	\caption{Long-time $H^1$-error for initial data in $H^2$ ($\theta=2	$) with fixed $\tau=0.05$.}
	\label{fig:cubic_err_fn_of_epsilon}
	%\label{fig:rational_end_times}
\end{figure}

Furthermore we see in Figure~\ref{fig:cubic_err_fn_of_t} that the new methods lead to much improved error over long-times even when compared to splitting methods.

\begin{figure}[h!]
	\centering
	\includegraphics[width=0.5\textwidth]{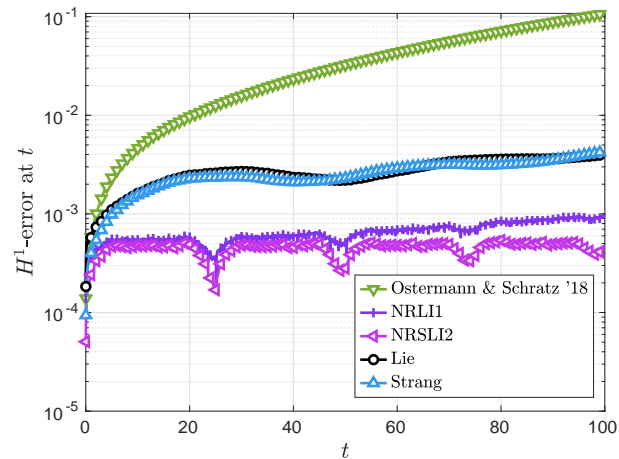}
	\caption{Error as a function of $t$ for initial data in $H^2$ ($\theta=2$) with fixed $\tau=0.01$.}
	\label{fig:cubic_err_fn_of_t}
\end{figure}

Finally, we can verify in Figure~\ref{fig:cubic_dependence_on_tau} that our methods are indeed true low-regularity integrators, which match the linear and quadratic convergence rates predicted in Theorems~\ref{thm:eb_nl} \& \ref{thm:eb_w2}.

\begin{figure}[h!]
	\centering
	\begin{subfigure}{0.495\textwidth}
		\centering
		\includegraphics[width=0.985\textwidth]{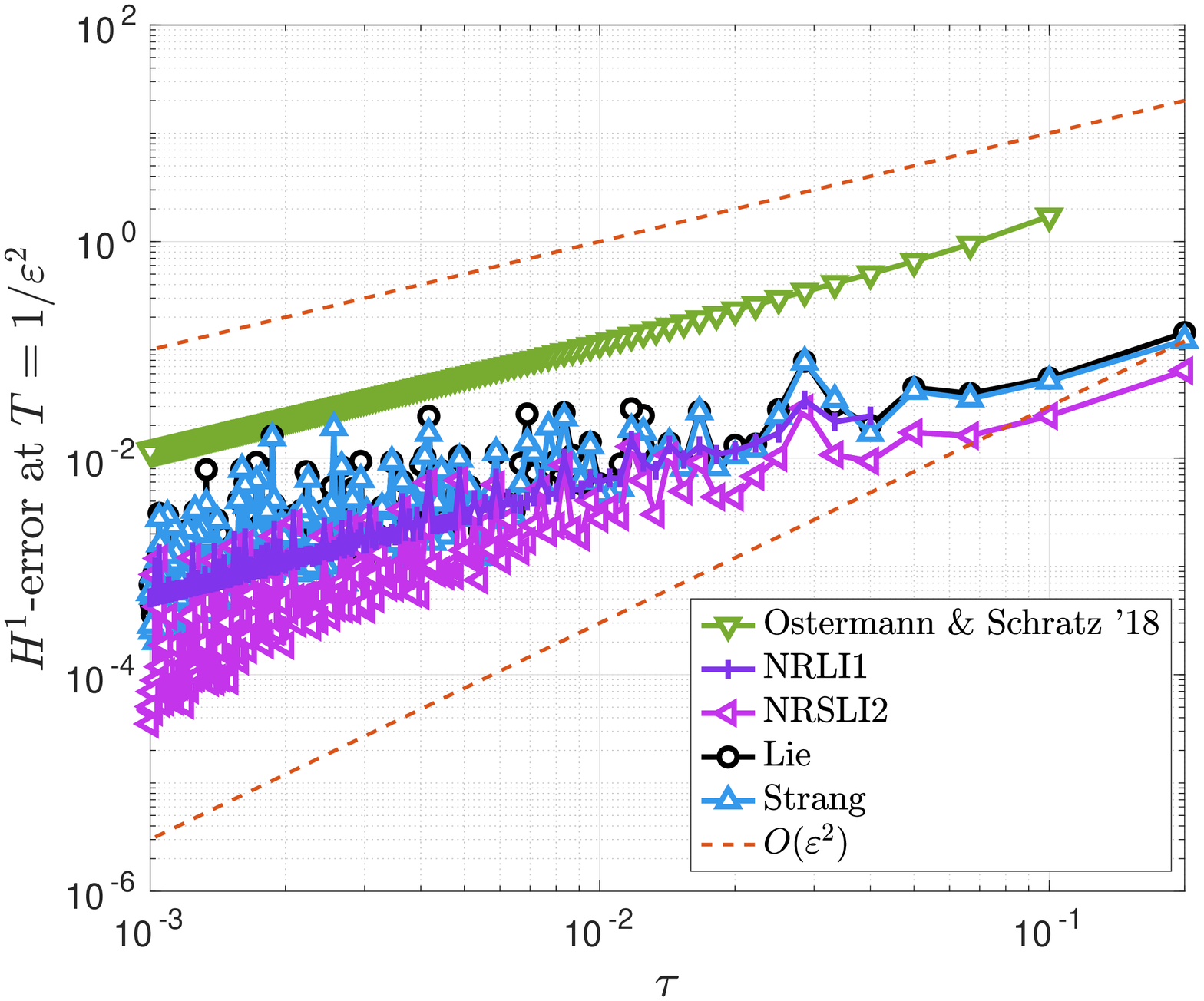}
		\caption{$H^2$ initial data ($\theta=2$).}
	\end{subfigure}

	\begin{subfigure}{0.495\textwidth}
	\centering
	\includegraphics[width=0.985\textwidth]{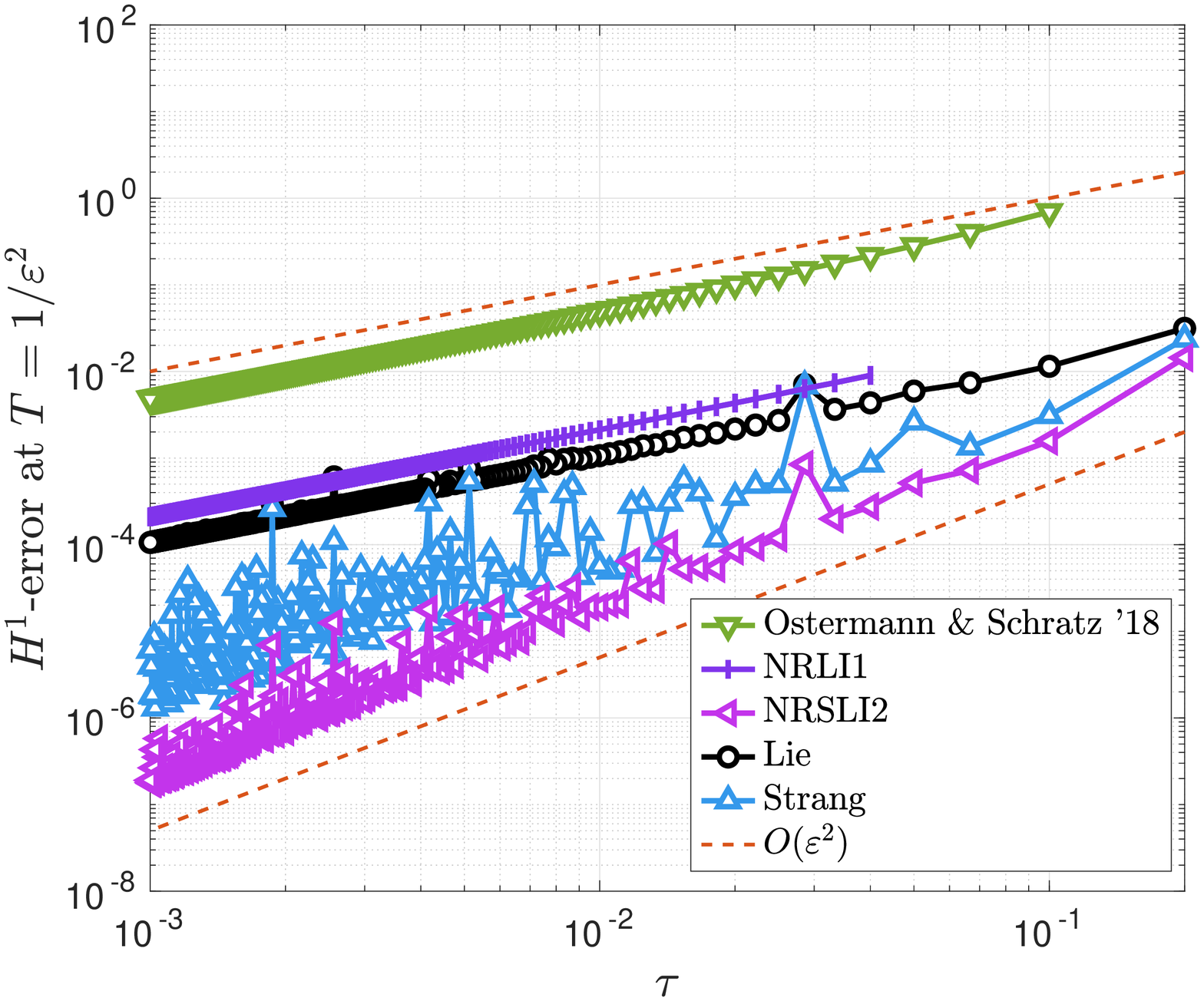}
	\caption{$H^3$ initial data ($\theta=3$).}
\end{subfigure}
	\begin{subfigure}{0.495\textwidth}
		\centering
		\includegraphics[width=0.985\textwidth]{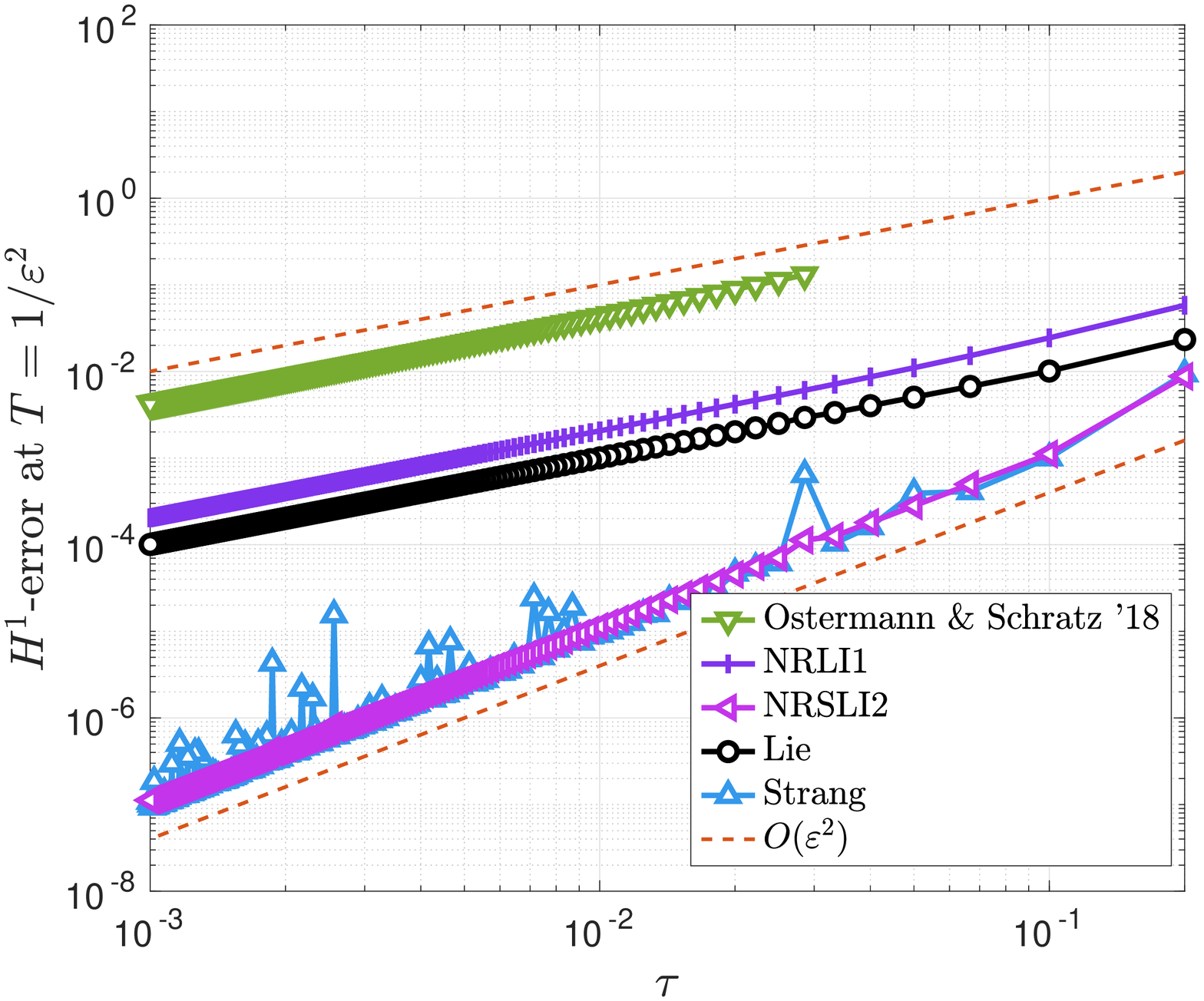}
		\caption{$H^4$ initial data ($\theta=4$).}
	\end{subfigure}
	\caption{Long-time $H^1$-error as a function of $\tau$ for fixed $\varepsilon=0.25$.}
	\label{fig:cubic_dependence_on_tau}
\end{figure}

\section{Conclusions}
Improved uniform error bounds on different low-regularity schemes for the long-time dynamics of the nonlinear Schr\"odinger (NLSE) with weak nonlinearity or small initial data were rigorous established. For the quadratic NLSE, the improved uniform $H^r$ bounds with $r>1/2$ for the first-order scheme and symmetric second-order scheme up to the time of order  $O(1/\eps)$ were carried out at $O(\eps\tau)$ and $O(\eps\tau^2)$ for the solution in $H^r$, respectively. For the cubic NLSE, we designed new non-resonant first-order and symmetric second-order low-regularity schemes and established the improved uniform error bounds with the help of the regularity compensation oscillation (RCO) technique up to the time of order $O(1/\eps^2)$. Numerical results were presented to confirm the improved uniform error bounds and underline the improved performance of our novel non-resonant schemes.

%%%%%% Bibliography  %%%%%%

\end{document}